\documentclass[a4paper]{article}
\usepackage[utf8]{inputenc}
\usepackage{booktabs}
\usepackage{amssymb,amsfonts,amsthm,amsmath}
\usepackage{subcaption}
\usepackage{graphicx}
\usepackage{url}
\usepackage{geometry}

\begingroup\expandafter\expandafter\expandafter\endgroup
\expandafter\ifx\csname pdfsuppresswarningpagegroup\endcsname\relax
\else
\fi

\newcommand{\bigoh}{{\mathcal O}}
\newcommand{\C}{{\mathcal C}}
\newcommand{\Complex}{{\mathbb C}}
\newcommand{\Dt}{{\Delta t}}

\newcommand{\R}{{\mathbb R}}

\theoremstyle{plain}
  \newtheorem{theorem}{Theorem}
  \newtheorem{lemma}[theorem]{Lemma}
  \newtheorem{corollary}[theorem]{Corollary}

\theoremstyle{definition}
  \newtheorem{remark}{Remark}

\DeclareMathOperator{\sech}{sech}

\title{Relaxation Runge-Kutta Methods: Conservation and stability for
Inner-Product Norms}
\date{}
\author{
David I. Ketcheson\thanks{Computer, Electrical, and Mathematical Sciences \& Engineering Division,
King Abdullah University of Science and Technology, 4700 KAUST, Thuwal
23955, Saudi Arabia. (david.ketcheson@kaust.edu.sa)}}

\begin{document}

\maketitle

\begin{abstract}
We further develop a simple modification of Runge--Kutta methods that guarantees
conservation or stability with respect to any inner-product norm.  The modified
methods can be explicit and retain the accuracy and stability properties of the
unmodified Runge--Kutta method.  We study the properties of the modified
methods and show their effectiveness through numerical examples, including
application to entropy-stability for first-order hyperbolic PDEs.
\end{abstract}

\sloppy

\section{Motivation and background}

Consider the initial value problem
\begin{subequations} \label{IVP}
\begin{align}
    u'(t) & = f(t,u(t)) \\
    u(t_0) & = u_0
\end{align}
\end{subequations}
where\footnote{We consider real spaces for simplicity, but the methods
developed here are also applicable in complex spaces.} $u:\R \to \R^m$ and $f:
\R\times\R^m \to \R^m$.  In this work we focus on problems that are dissipative
with respect to some inner-product norm:
\begin{align} \label{dissipativity}
\frac{d}{dt}\|u(t)\|^2 = 2\langle u, f(t,u) \rangle \le 0.
\end{align}
Here and throughout, $\langle \cdot, \cdot \rangle$
denotes an inner product and $\|\cdot\|$ the corresponding norm;
we will sometimes refer to $\|u\|^2$ as {\em energy}.
For dissipative problems, it is desirable that the numerical solution
mimic \eqref{dissipativity}:
\begin{align} \label{monotonicity}
    \|u^{n+1}\| \le \|u^n\|.
\end{align}
Herein, a method is called {\em monotonicity preserving}\/ if it guarantees
\eqref{monotonicity} for all problems satisfying \eqref{dissipativity}.

We say the problem \eqref{IVP} is {\em conservative} if
\begin{align} \label{conservative}
    \langle u, f(t,u) \rangle = 0.
\end{align}
For such problems it is desirable to discretely conserve energy:
\begin{align} \label{discrete_conservation}
    \|u^{n+1}\| = \|u^n\|.
\end{align}
A method is also called {\em conservative}\/ if it guarantees
\eqref{discrete_conservation} for all problems satisfying \eqref{conservative}.


In many applications, numerical conservation or monotonicity preservation
are of great importance.  Not only do these properties guarantee that the solution
remains bounded, but violation of these properties can lead to solutions
that are unphysical and qualitatively wrong.  Nevertheless, most numerical
methods do not enforce these properties exactly, but only up to truncation
errors. This is particularly true for explicit Runge--Kutta methods.
Even if one considers only linear autonomous systems of ODEs, it has been shown
that many explicit Runge--Kutta methods --- including the classical 4th-order
method --- are not even conditionally
monotonicity preserving; see \cite{tadmor2003entropy,sun2017,ranocha2018l_2,sun2018strong}.  
For general nonlinear autonomous systems, no methods of
order greater than one are known to be monotonicity preserving \cite{ranocha2018strong}.
For conservative problems, no explicit Runge--Kutta method of
any order can enforce discrete conservation, even for linear autonomous
problems.  Formulas for the production of a general convex entropy function
(which includes the special case of an inner-product norm) have been
derived in \cite{lozano2018entropy}.

Monotonicity preservation for Runge--Kutta methods with respect to an
inner-product norm was studied by
Higueras \cite{higueras2005}, with results closely connected to an earlier
study of contractivity preservation by Dahlquist \& Jeltsch\cite{dahlquist2006}.
However, to obtain results for explicit RK methods in that framework, it is
necessary to require strict inequality in \eqref{dissipativity},
making the results inapplicable for conservative problems or
for typical high-order semi-discretizations of hyperbolic PDEs.

Unconditionally conservative and monotonicity-preserving methods exist, for
instance, among the classes of implicit Runge--Kutta
and (for Hamiltonian systems) partitioned Runge--Kutta methods; see
\cite{hairer2006geometric} and references therein.

While classical explicit Runge--Kutta or linear multistep methods cannot
preserve general quadratic invariants,
some explicit energy-conservative methods have been developed
by going outside these traditional classes (as is done also in the
present work).
Indeed, it is possible to modify any Runge--Kutta method to preserve energy
or other first integrals
by a technique known as projection; see e.g. \cite{Grimm2005,calvo2010projection}.
In this approach, at the end of each time step the solution is
projected onto a desired set in order to ensure some property like
conservation or monotonicity.
The approach used in the current work can be viewed as a projection
method where the projection is performed along a direction corresponding
to the next time step update, but with an additional modification that the
size of the time step is also modified.  This approach was originally proposed
by Dekker \& Verwer \cite[pp. 265-266]{dekker1984}, who noted that the classical
four-stage RK method could be modified slightly to conserve energy
while maintaining its accuracy.
The idea was extended in \cite{delbuono2002explicit} to a restricted class of
fourth order methods.
This was developed further in \cite{calvo2006preservation}
by giving a general proof that applying the technique (without the step size
adjustment) to a RK method of order $p$ results in a method of order at least
$p-1$.  Subsequent development in \cite{calvo2006preservation,
calvo2010projection, laburta2015numerical} focuses on {\em embedded projection}
methods; i.e. methods that project in a direction determined by an embedded
Runge--Kutta pair.  The approach described in the present work could be
viewed as a variant in which the ``embedded'' method is simply the identity
map, but with an additional twist that requires reinterpretation of
the new step solution as an approximation at a slightly different time.

Like embedded projection methods, the methods proposed here are:
\begin{itemize}
    \item explicit
    \item conditionally conservative and monotonicity-preserving for general
            nonlinear ODEs
    \item arbitrarily high order accurate
    \item linearly covariant
\end{itemize}
Furthermore, they do not require partitioning or temporal staggering, and they
inherit other useful properties (such as strong stability preservation) of a
selected Runge--Kutta method.  Preservation of more general (non-inner-product)
functionals is also possible with modification similar to that described
herein; see \cite{paper2}.

The main contributions of the present work are: first, to further develop
these methods that, while not entirely new, seem to have been overlooked;
second, to put them on a rigorous footing in terms of accuracy and stability
properties; and third, to explore their properties through analysis and
numerical experiments.

\subsection{Energy evolution by Runge--Kutta methods}
A Runge--Kutta method applied to \eqref{IVP} takes the form
\begin{subequations} \label{RKM}
\begin{align}
    y_i & = u^n + \Dt \sum_{j=1}^s a_{ij} f(t_n+c_j \Dt, y_j) \label{stages} \\
    u(t_n + \Dt) \approx u^{n+1} & = u^n + \Dt \sum_{j=1}^s b_j f(t_n+c_j \Dt, y_j). \label{ynp1}
\end{align}
\end{subequations}
We make the usual assumption that $c_j = \sum_i a_{ij}$.
For convenience we introduce the shorthand
$$
f_i = f(t_n+c_i\Dt, y_i)
$$
for the $i$th stage derivative.  The change in energy from one step to the next is
\begin{align*}
\|u^{n+1}\|^2 - \|u^n\|^2 & = \left\|u^n + \Dt\sum_{j=1}^s b_j f_j \right\|^2 - \|u^n\|^2 \\
& = 2\Dt\sum_{j=1}^s b_j \langle u^n, f_j \rangle  + \Dt^2\sum_{i,j=1}^s b_i b_j \langle f_i, f_j \rangle \\
& = 2 \Dt\sum_{j=1}^s b_j \langle y_j, f_j \rangle - 2\Dt\sum_{j=1}^s b_j \langle u^n-y_j, f_j \rangle + \Dt^2\sum_{i,j=1}^s b_i b_j \langle f_i, f_j \rangle,
\end{align*}
which can be rewritten using \eqref{stages} as
\begin{align} \label{energy_update}
\|u^{n+1}\|^2 - \|u^n\|^2 & = 2\Dt\sum_{j=1}^s b_j \langle y_j, f_j \rangle - 2\Dt^2\sum_{i,j=1}^s b_i a_{ij} \langle f_j, f_i \rangle + \Dt^2\sum_{i,j=1}^s b_i b_j \langle f_i, f_j \rangle.
\end{align}
The first sum on the right side of \eqref{energy_update} is zero for
conservative systems, and it is negative for dissipative systems if $b_j\ge 0$
for all $j$.
However, the remaining two terms may lead to violation of the conservation
or monotonicity property.
Those two terms can be written together as the bilinear form
\begin{align} \label{mF}
    - \Dt^2 \sum_{i,j=1}^s m_{ij} \langle f_i, f_j \rangle,
\end{align}
where, letting $B$ denote the diagonal matrix with entries $b_j$,
$$
M = BA + A^T B - bb^T.
$$
This is the traditional analysis used in studying symplecticity, algebraic
stability, and other properties of RK methods
(see e.g. \cite{burrage1979stability,dekker1984,hairer2006geometric}).  If the matrix $M$
is positive semidefinite and the weights are
nonnegative, then the method is said to be algebraically stable.  Clearly, such methods are
unconditionally monotonicity-preserving.  If $M=0$, the method is said to be symplectic; clearly such methods
are unconditionally conservative.  Certain well-known implicit methods have these properties.
However, explicit methods cannot be algebraically stable or symplectic.

\section{Relaxation Runge--Kutta methods}
The relaxation version of the method \eqref{RKM} is obtained by replacing
\eqref{ynp1} with the update formula
\begin{align}
    u(t_n+\gamma_n \Dt) \approx u^{n+1}_\gamma & = u^n + \Dt \gamma_n \sum_{j=1}^s b_j f(t_n+c_j \Dt, y_j). \label{ynp1_gamma}
\end{align}
The only difference between \eqref{ynp1_gamma} and \eqref{ynp1} is the factor
$\gamma_n$ that multiplies the step size.
We can think of the original Runge--Kutta method \eqref{RKM} as determining only the
direction in which the solution will be updated, while the choice of $\gamma_n$
determines how far to step in that direction.  From this point of view
$\gamma_n$ is similar to the relaxation parameter used in some iterative
algebraic solvers, and for this reason we refer to these methods as relaxation
Runge--Kutta (RRK) methods.

With this change, \eqref{energy_update} becomes
\begin{align}
\begin{split} \label{energy_update_gamma}
\|u^{n+1}_\gamma\|^2 - \|u^n\|^2 = 2\gamma_n \Dt\sum_{j=1}^s b_j \langle y_j, f_j \rangle
        & - 2\gamma_n \Dt^2\sum_{i,j=1}^s b_i a_{ij} \langle f_j, f_i \rangle \\ & + \gamma_n^2 \Dt^2\sum_{i,j=1}^s b_i b_j \langle f_i, f_j \rangle.
\end{split}
\end{align}
We can eliminate the last two terms by setting
\begin{align} \label{gammadef}
\gamma_n = \frac{2 \sum_{i,j=1}^s b_i a_{ij} \langle f_i,f_j \rangle}{\sum_{i,j=1}^s b_i b_j \langle f_i, f_j \rangle},
\end{align}
so that
\begin{align*}
\|u^{n+1}_\gamma\|^2 - \|u^n\|^2 & = 2\gamma_n \Dt\sum_{j=1}^s b_j \langle y_j, f_j \rangle.
\end{align*}
In case the the denominator of \eqref{gammadef} vanishes, we have
$u^{n+1}=u^n$, so we can achieve conservation or monotonicity by taking simply
$\gamma_n = 1$.  We thus define in place of \eqref{gammadef}
\begin{align} \label{gammadefreal}
\gamma_n & = \begin{cases} 1 & \| \sum_{j=1}^s b_j f_j \|^2 = 0  \\
        \frac{2 \sum_{i,j=1}^s b_i a_{ij} \langle f_i,f_j
        \rangle}{\sum_{i,j=1}^s b_i b_j \langle f_i, f_j \rangle} & \|
        \sum_{j=1}^s b_j f_j \|^2 \ne 0.
        \end{cases}
\end{align}

Because we will interpret $u_{\gamma}^{n+1}$ as an approximation to the solution
at time $u(t_n+\gamma_n \Dt)$, it is important that $\gamma_n >0$.  It is straightforward
to show the following:
\begin{lemma} \label{posgamlem}
    Let $\sum b_i a_{ij} >0$, let $f$ be sufficiently smooth, and let $\gamma_n$
    be defined by \eqref{gammadefreal}.  Then $\gamma_n>0$ for sufficiently small $\Dt>0$.
\end{lemma}
Note that the condition in Lemma \ref{posgamlem} holds for all methods of
order two or higher.

\begin{remark}
A more detailed analysis indicates that
{\em small enough} here means that $\Dt$ should be no more than about $1/L$,
where $L$ is the Lipschitz constant of $f$.  This is also the order of the
absolutely stable step size when using any explicit method, so one may expect
that using an absolutely stable step size will also yield $\gamma_n>0$.
\end{remark}

For conservative systems this gives exact energy conservation; for dissipative
systems it preserves dissipativity as long as all weights are non-negative.
\begin{theorem} \label{stabilitythm}
    Let $(A,b)$ be the coefficients of a Runge--Kutta method of order at least two.
    The corresponding relaxation Runge--Kutta method defined by \eqref{stages}
    and \eqref{ynp1_gamma} with $\gamma_n$ defined by \eqref{gammadefreal} is conservative.
    If the weights are nonnegative, then the relaxation method is also monotonicity preserving
    as long as $\Dt$ is chosen so that $\gamma_n \ge 0$.
\end{theorem}

\begin{remark}
  The relaxation RK method still conserves linear invariants, which is important
  for instance in the semi-discretization of hyperbolic conservation laws.
\end{remark}

\begin{remark}
    A version of Theorem \ref{stabilitythm} applicable only to conservative systems
    appears in \cite[Thm. 2.1]{calvo2006preservation}, along with a formula for
    $\gamma_n$ that is more computationally efficient (but correct only for
    conservative systems).
\end{remark}

\begin{remark} \label{ipremark}
    The denominator of the expression for $\gamma_n$ in \eqref{gammadef}
    is simply the norm of the step update, and can be computed with a
    single inner product.
    It has been pointed out by Hendrik Ranocha that
    by solving \eqref{energy_update} for the term that is the numerator of the expression
    for $\gamma_n$, it can be computed using only $s$ inner products \cite{ranocha_communication}.
    Thus $\gamma_n$ can be computed with just $s+1$ inner products.
\end{remark}

The update formula \eqref{ynp1_gamma} is equivalent to replacing the
coefficients $b_j$ in the Runge--Kutta method
with $\gamma_n b_j$.  It can also be viewed roughly as taking a Runge--Kutta step of size
$\gamma_n \Dt$ in place of $\Dt$, but notice that the original step size $\Dt$ is
still used in the computation of the stages $y_i$.  Both viewpoints
(rescaling $b$ and rescaling $\Dt$) will be useful in our analysis.

We will see below in Lemma~\ref{gammalem} that, for reasonable values of $\Dt$, $\gamma_n$ is close to unity.
Although the proof of this fact is technical, the result itself
is not surprising, at least for conservative systems.  For such systems, the
value of $\gamma_n$ given by \eqref{gammadefreal} is a solution of
$$
\Delta E(\gamma) := \|u^{n+1}_\gamma\|^2 - \|u^{n+1}\|^2 = 0,
$$
which is a quadratic function of $\gamma$ with one root at zero.
Given that $\Delta E(1) = \bigoh(\Dt^{p+1})$ and that according to \eqref{energy_update_gamma}
$\Delta E \propto \Dt^2$ for conservative problems, it is natural to expect that $\Delta E(\gamma)$ has
a zero within $\bigoh(\Dt^{p-1})$ of unity.
The method described here can be viewed as using a line search to determine
the step size that solves $\Delta E(\gamma)=0$ and thus exactly conserves energy.

For dissipative systems also, we will see that the last two terms
in \eqref{energy_update} are not important (in the sense of accuracy) to the numerical
approximation of the energy evolution; i.e. the term
$2\Dt\sum_j b_j \langle y_j, f_j \rangle$ approximates the energy evolution
over one step to order $p$:
$$
\Delta E(1) = \|u(t_n+\Dt)\|^2 - \|u(t_n)\|^2 = 2\Dt\sum_j b_j \langle y_j, f_j \rangle + \bigoh(\Dt^{p+1}).
$$

\subsection{Accuracy}
At first glance, the RRK method (given by \eqref{stages} with \eqref{ynp1_gamma})
seems to be not even consistent, since $\sum_j \gamma_n b_j = \gamma_n \ne 1$ in
general. For a classical RK method, the condition $\sum_j b_j = 1$, along with higher order conditions, is
necessary for local consistency of a given order.  But for an RRK method, because
the coefficients depend on $\Dt$, we can still obtain high order accuracy if the
order conditions are {\em nearly} satisfied, which is true if $\gamma_n$ is sufficiently close to unity.

\begin{theorem} \label{orderthm1}
Let $(a_{ij}, b_j)$ be the coefficients of a Runge--Kutta method of order $p$.
Consider the RRK method defined by \eqref{stages}, \eqref{ynp1_gamma}
and suppose that
\begin{align} \label{gamma1}
    \gamma_n = 1 + \bigoh(\Dt^{p-1}).
\end{align}
Then:
\begin{enumerate}
    \item (IDT method) If the solution $u^{n+1}_\gamma$ is interpreted as an approximation to $u(t_n + \Dt)$,
        the method has order $p-1$.
    \item (RRK method) If the solution $u^{n+1}_\gamma$ is interpreted as an approximation to $u(t_n + \gamma_n \Dt)$,
        the method has order $p$.
\end{enumerate}
\end{theorem}
\begin{proof}
First we note that by taking appropriate linear combinations of the usual
order conditions, the conditions for a RK method to have order $p$ can be
written as (see \cite{albrecht1996})
\begin{subequations} \label{albrechtOCs}
\begin{align}
    b^T c^{k-1} - 1/k & = 0 & 1 \le k \le p \label{bushytrees} \\
    b^T v & = 0 & \forall v \in V_p,
\end{align}
\end{subequations}
where $V_p$ is a set of vectors depending only on $a_{ij}$ whose specific
elements are not important here.  For the RRK method, we replace $b$ with $\gamma_n b$,
obtaining the conditions
\begin{subequations} \label{albrechtOCsgamma}
\begin{align}
    \gamma_n b^T c^{k-1} - 1/k & = 0 & 1 \le k \le p \label{bushytreesgamma} \\
    \gamma_n b^T v & = 0 & \forall v \in V_p. \label{otherOCsgamma}
\end{align}
\end{subequations}
Given a method that satisfies \eqref{albrechtOCs}, clearly \eqref{otherOCsgamma} is satisfied as well,
while the left hand side of \eqref{bushytreesgamma} is $\bigoh(\Dt^{p-1})$.
In the error expansion, this value gets multiplied by $\bigoh(\Dt^k)$, so the leading error term
is $\bigoh(\Dt^p)$ and the method has order $p-1$.

To prove the second part of Theorem \ref{orderthm1}, we view the solution given by \eqref{ynp1_gamma} as
an interpolant for the RK solution, evaluated at $\gamma_n \Dt$.  The order
conditions for this interpolated solution are
\begin{subequations}
\begin{align}
    \gamma_n b^T c^{k-1} - \gamma_n^k/k & = 0 & 1 \le k \le p \label{bushy_gamma} \\
    \gamma_n b^T v & = 0 & \forall v \in V_p. \label{orthoOCs_gamma}
\end{align}
\end{subequations}
We see that the conditions \eqref{orthoOCs_gamma} and the first bushy-tree
condition (\eqref{bushy_gamma} with $k=1$) are still exactly fulfilled.
The remaining bushy-tree conditions \eqref{bushy_gamma}
are not exactly satisfied; for $k\ge 2$ we have, using \eqref{bushytrees} and \eqref{gamma1},
$$
\gamma_n (b^T c^{k-1} - \gamma_n^{k-1}/k) = \frac{\gamma_n}{k}(1-(1+\Dt^{(k-1)(p-1)})) = \bigoh(\Dt^{(k-1)(p-1)}).
$$
In the error expansion, each of these residuals is multiplied by
$\bigoh(\Dt^k)$, so that the overall error incurred is $\bigoh(\Dt^{(k-1)p+1})$.
The exponent in this expression is at least $p+1$ for $k\ge 2$.
\end{proof}

It turns out that $\gamma_n$ satisfies the condition required in Theorem~\ref{orderthm1}.
\begin{lemma} \label{gammalem}
Let $a_{ij}, b_j$ denote the coefficients of a Runge--Kutta method \eqref{RKM} of
order $p$, let $f$ be a sufficiently smooth function, and let $\gamma_n$ be
defined by \eqref{gammadefreal}.  Then
\begin{align*}
    \gamma_n = 1 + \bigoh(\Dt^{p-1}).
\end{align*}
\end{lemma}

We illustrate Lemma \ref{gammalem} in Figure \ref{fig:gamma}, which shows
the convergence of $\gamma_n$ to 1 as the step size is reduced for the problem
\eqref{prob1}.  Due to the symmetry of the problem, even faster convergence
is observed for some methods.

\begin{figure}
    \centering
    \includegraphics[width=4in]{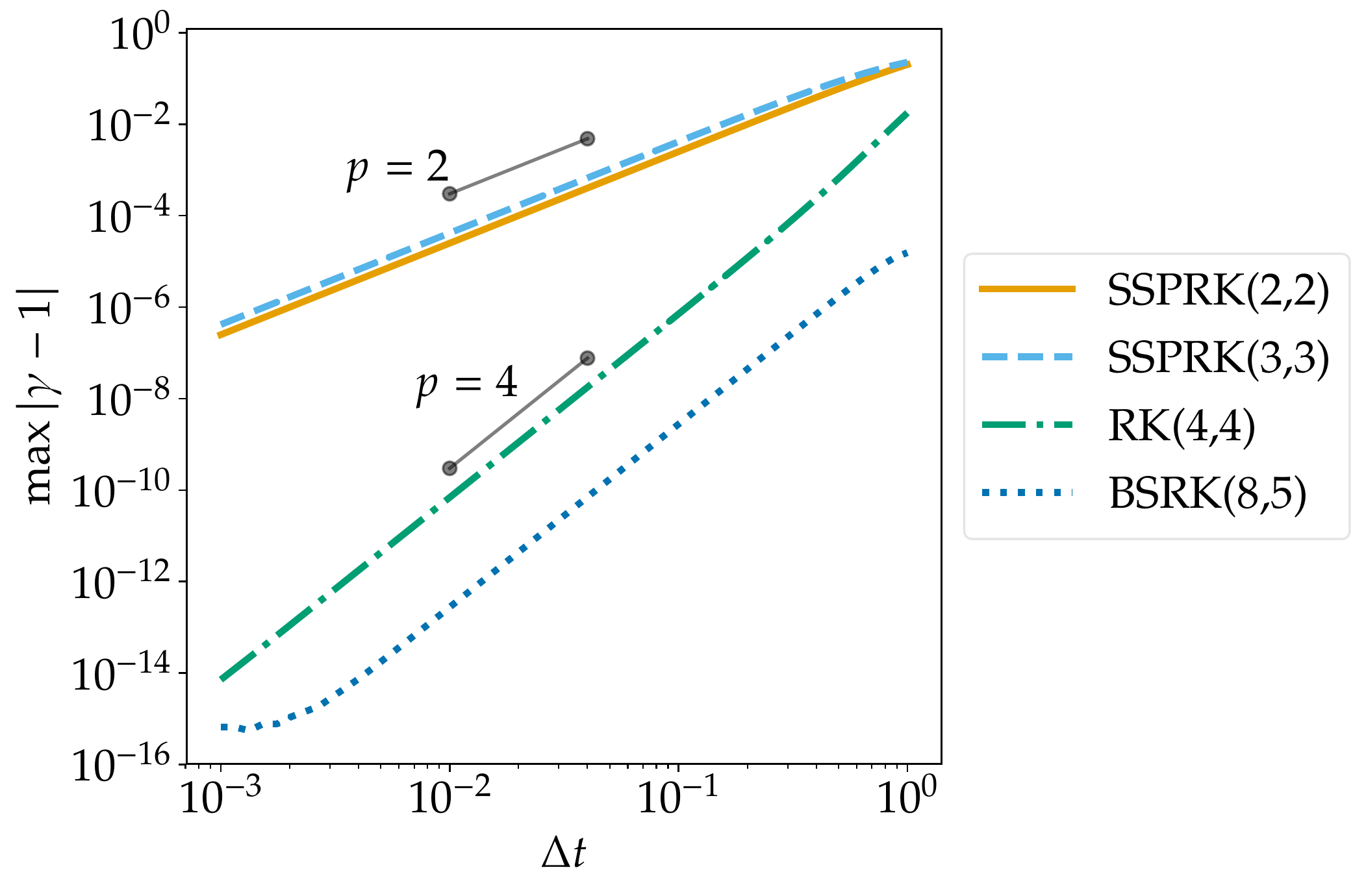}
    \caption{Convergence of $\gamma_n$ to unity for various methods applied to
             problem \eqref{prob1}.  The observed rate is at least the expected
             $\bigoh(\Dt^{p-1})$ for all methods.  The SSPRK(2,2) and RK(4,4)
             methods show convergence at one order higher than expected.
            \label{fig:gamma}
    }
\end{figure}

\begin{remark}
    Lemma \ref{gammalem} was proved in \cite[Proposition 4]{delbuono2002explicit}
    for the special case of a certain family of four-stage, fourth-order methods.
\end{remark}

Before proving Lemma~\ref{gammalem} we state the main consistency result, which
follows immediately from Theorem~\ref{orderthm1} and Lemma~\ref{gammalem}.

\begin{corollary} \label{ordercor}
Let $(A, b)$ be the coefficients of a Runge--Kutta method of order $p$, and
consider the RRK method defined by \eqref{stages} and \eqref{ynp1_gamma}
with $\gamma_n$ defined by \eqref{gammadefreal}.
\begin{itemize}
    \item (IDT method) If the solution $u^{n+1}_\gamma$ is interpreted as an approximation to $u(t_n + \Dt)$,
        the method has order $p-1$.
    \item (RRK method) If the solution $u^{n+1}_\gamma$ is interpreted as an approximation to $u(t_n + \gamma_n \Dt)$,
        the method has order $p$; i.e.\ the solution after one step satisfies
        $$
        \|u^1_\gamma - u(\gamma_1 \Dt)\| = \bigoh(\Dt^{p+1}).
        $$
\end{itemize}
\end{corollary}

\begin{remark}
    The first part of Corollary \ref{ordercor} was proved by different means in
    \cite[Thm. 2.1 (ii)]{calvo2006preservation}.
\end{remark}

In order to prove Lemma~\ref{gammalem}
we use the theory of B-series and follow the notation of \cite{HW1}.
In the remainder of this section, $t$ denotes a tree rather than time
and $\gamma(t)$ denotes the density of a tree rather than the relaxation
step length (see \cite[Dfn. 2.10]{HW1}).
Let $\rho(t)$ denote the order of tree $t$, and let $t'$ denote
the tree obtained by attaching a new root node to the root of $t$.
Let $[t_1, t_2]$ denote the tree obtained by attaching a new root node
to the root nodes of $(t_1, t_2)$ as in  \cite[Dfn. 2.12]{HW1}.
Then $$\rho(t')= \rho(t)+1$$ and
$$\gamma([t_1,t_2]) = \gamma(t_1)\gamma(t_2) (\rho(t_1)+\rho(t_2)+1).$$
Finally, let $\Phi_j(t)$ be defined as in \cite[Dfn. 2.9]{HW1}.
We recall \cite[Thm. 2.13]{HW1}:
\begin{theorem}
    A Runge--Kutta method is of order $p$ iff
    \begin{align} \label{OC}
        \sum_{j=1}^s b_j \Phi_j(t) = \frac{1}{\gamma(t)}
    \end{align}
    for all trees of order $\le p$.
\end{theorem}

\begin{proof}[Proof of Lemma~\ref{gammalem}]
We can write $\gamma = 1-\delta/\eta$ where $\eta=\bigoh(1)$ and
\begin{align} \label{deltadef}
\delta = \sum_i b_i \sum_j (b_j - 2 a_{ij}) \langle f_i, f_j \rangle.
\end{align}
Thus it suffices to show that
$$
\sum_i b_i \sum_j (b_j - 2 a_{ij}) \langle f_i, f_j \rangle = \bigoh(\Dt^{p-1}).
$$
From \cite[Thm. II.2.11]{HW1} we have the Taylor series for the $i$th stage
derivative:
$$
f_i = \sum_{q=1}^\infty \frac{\Dt^q}{q!} \sum_{t\in LT_q} \gamma(t) \sum_k a_{ik} \Phi_k(t) F(t)(u_0),
$$
where $F(t)$ is the elementary differential corresponding to $t$.
Thus $\langle f_i, f_j \rangle$ can be expressed as a linear combination of
inner products of elementary differentials:
$$
\langle f_i, f_j \rangle = \sum_{t_1} \sum_{t_2} \frac{\Dt^{\rho(t_1)+\rho(t_2)}}{\rho(t_1)!\rho(t_2)!}\gamma(t_1) \gamma(t_2) \beta_{ij}(t_1,t_2) \langle F(t_1)(u_0), F(t_2)(u_0) \rangle,
$$
where $\rho(t)$ is the order (number of nodes) of tree $t$, and $t_1, t_2$ range over the
set of all labelled rooted trees.  Here
\begin{align*}
    \beta_{ij}(t_1, t_2) & = \sum_k a_{ik} \Phi_k(t_1) \sum_l a_{jl} \Phi_j(t_2) \\
    & = \Phi_i(t_1')  \Phi_j(t_2').
\end{align*}

Because of the symmetry of the inner product, it is sufficient to show that
$$
 \sum_i b_i \sum_j (b_j - 2 a_{ij})   (\beta_{ij}(t_1,t_2) + \beta_{ij}(t_2, t_1)) = 0
$$
for all pairs of trees $(t_1,t_2)$ satisfying
\begin{align} \label{treeorder}
    \rho(t_1)+\rho(t_2)\le p-2.
\end{align}


We have
\begin{align*}
\sum_i b_i \sum_j (b_j -2 a_{ij}) & (\beta_{ij}(t_1,t_2) + \beta_{ij}(t_2,t_1)) \\ & = \sum_i b_i \sum_j (b_j-2a_{ij})\left(\Phi_i(t_1') \Phi_j(t_2') + \Phi_i(t_2')\Phi_j(t_1')\right) \\
    & = \frac{2}{\gamma(t_1')\gamma(t_2')} - 2 \sum_i b_i \left(\Phi_i(t_1')\Phi_i(t_2'') + \Phi_i(t_2')\Phi_i(t_1'')\right) \\
    & = \frac{2}{\gamma(t_1')\gamma(t_2')} - 2 \sum_i b_i \left(\Phi_i([t_1,t_2']) + \Phi_i([t_2,t_1'])\right) \\
    & = \frac{2}{\gamma(t_1')\gamma(t_2')} - \frac{2}{\rho(t_1)+\rho(t_2)+2}\left(\frac{1}{\gamma(t_1)\gamma(t_2')} \frac{1}{\gamma(t_2)\gamma(t_1')}\right) \\
    & = \frac{2}{\gamma(t_1')\gamma(t_2')}\left( 1 - \frac{1}{\rho(t_1)+\rho(t_2)+2} \left( \frac{\gamma(t_1')}{\gamma(t_1)} + \frac{\gamma(t_2')}{\gamma(t_2)}\right)\right) = 0.
\end{align*}
Here we have applied \eqref{OC} in various places, using the fact that
our method is of order $p$ and that (due to \eqref{treeorder})
the trees in question all have order less than or equal to $p$.
\end{proof}

\begin{remark}
    It is possible to prove a generalization of Lemma \ref{gammalem}
    without the use of B-series; see \cite{paper2}.  The proof above
    is included here to show a direct approach that is very different from that
    used in \cite{paper2}.
\end{remark}

\subsection{Comparison with projection methods}
A projection Runge--Kutta method \cite[Section IV.4]{hairer2006geometric}
consists of the traditional Runge--Kutta
formula \eqref{RKM} followed by a projection step:
$$
  u_{\lambda}^{n+1} = u^{n+1} - \lambda \Phi,
$$
where $\Phi$ is the projection direction and $\lambda$ is chosen so
that $u_\lambda^{n+1}$ lies on a desired manifold.
A relaxation Runge--Kutta method can be viewed as a projection method
along the direction $\Phi=u^{n+1}-u^n$ with step length
$\lambda = 1-\gamma_n$.  However, the existing formalism for such methods
does not include the possibility of interpreting the new solution as
an approximation at a time different from $t_n + \Dt$.
In \cite{calvo2006preservation}, this projection perspective was applied
and it was shown that the resulting method is indeed of order $p-1$
when the original RK method is order $p$.  We follow the terminology
of \cite{calvo2006preservation} and refer to this interpretation as
the {\em incremental direction technique}, or IDT.

To properly write a relaxation method as a projection method we can write
\eqref{IVP} as the equivalent autonomous system of $m+1$ ODEs $v' = g(v)$
in the standard way with
\begin{align*}
v & = \begin{bmatrix} u \\ t \end{bmatrix}, & g(v) = \begin{bmatrix} f(v_{m+1},v_{1:m}) \\ 1 \end{bmatrix}.
\end{align*}
where $g(v) = f(v_{m+1},v_{1:m})$.  The projection of the solution of
this problem in the direction $\Phi=v^{n+1}-v^n$  requires
also projecting the updated value of $t$ to obtain the value $t_n+\gamma_n \Dt$.

\section{Stability properties of explicit relaxation RK methods}
All of the results in the previous section apply to general Runge--Kutta
methods.  In the rest of this work, we focus on explicit methods.

An important question not explicitly answered in the foregoing analysis is how large
the step size $\Dt$ can be taken in practice.
Theorem \ref{stabilitythm} guarantees unconditional stability for RRK methods
applied to any conservative problem, and guarantees stability for dissipative
problems as long as the step size is small enough.  Due to the overall explicit
nature of the methods, we should not expect in either case to obtain accurate
results using step sizes much larger than what the original explicit RK method
allows; in this respect RRK methods behave similarly to rational Runge--Kutta
methods, to which their form is also very similar
\cite{wambecq1978rational,hairer1980unconditionally}.  In this section we
investigate linear stability of RRK
methods and also study how the strong stability preserving (SSP) property is
affected by the use of RRK methods.

The behavior of an RRK method may be quite challenging to analyze, since $\gamma$
depends in a nonlinear way on the method coefficients and the numerical solution
itself.  On the other hand, over a single step we can think of $\gamma$ as a fixed
value close to unity, and study the properties of the Runge--Kutta method
\begin{align} \label{rkm_gamma}
\begin{array}{r|c}
c & A \\ \midrule & \gamma b^T
\end{array}.
\end{align}

\subsection{The stability function}
The stability function of a Runge--Kutta method is
$$
R(z) = 1 + zb^T(I-zA)^{-1}e,
$$
where $e$ is  vector with all entries equal to unity.
Thus the stability function corresponding to one step of the relaxation method \eqref{rkm_gamma} is
\begin{align}
R_\gamma(z) = 1 + z\gamma b^T(I-zA)^{-1}e.
\end{align}
Letting $\alpha_k$ denote the coefficients of $R(z)$ for an explicit method:
$$
    R(z) = 1 + \sum_{k=1}^s \alpha_k z^k,
$$
we have
$$
    R_\gamma(z) = 1 + \gamma \sum_{k=1}^s \alpha_k z^k.
$$
Let $S(A,b)\subset \Complex$ denote the region of absolute stability of
RK method $(A,b)$.  It turns out that $S(A,\gamma b)$ grows as $\gamma$ decreases.
\begin{theorem} \label{stabregthm}
	Let $\gamma_1, \gamma_2$ be given such that $0 \le \gamma_1 \le \gamma_2$.
        Then $S(A,\gamma_2 b)\subseteq S(A, \gamma_1 b)$.
\end{theorem}
\begin{proof}
    It is sufficient to show that if $|R_{\gamma_2}(z_*)|\le 1$ for some $z_*\in\Complex$
then $|R_{\gamma_1}(z_*)|\le 1$.  Let $w = \sum_{k=1}^s \alpha_k z_*^k$; then clearly
$|R_{\gamma_2}(z_*)| = |1+\gamma_2 w| \le 1$ implies $|R_{\gamma_1}(z_*)| = |1+\gamma_1 w| \le 1$.
\end{proof}

Theorem \ref{stabregthm} is illustrated in Figure \ref{fig:stabregs}, which shows the
stability region for relaxation RK methods with 2--4 stages and order equal to the stage number.
Stability region boundaries are shown for different values of $\gamma$ in the
range $[0.7, 1.3]$.  This is an exaggerated range compared to values of $\gamma$ that
are used in practice; for practical values of $\gamma$, the change in the stability region
is visually too small to notice.
It is particularly interesting to note that small reductions in $\gamma$ lead to significantly
enhanced stability along the imaginary axis.  Note also that as $\gamma\to 0$, the
RRK method tends to the identity map and $S(A,\gamma b) \to \Complex$.

\begin{figure}
    \centering
    \begin{subfigure}[t]{1.6in}
        \includegraphics[width=1.6in]{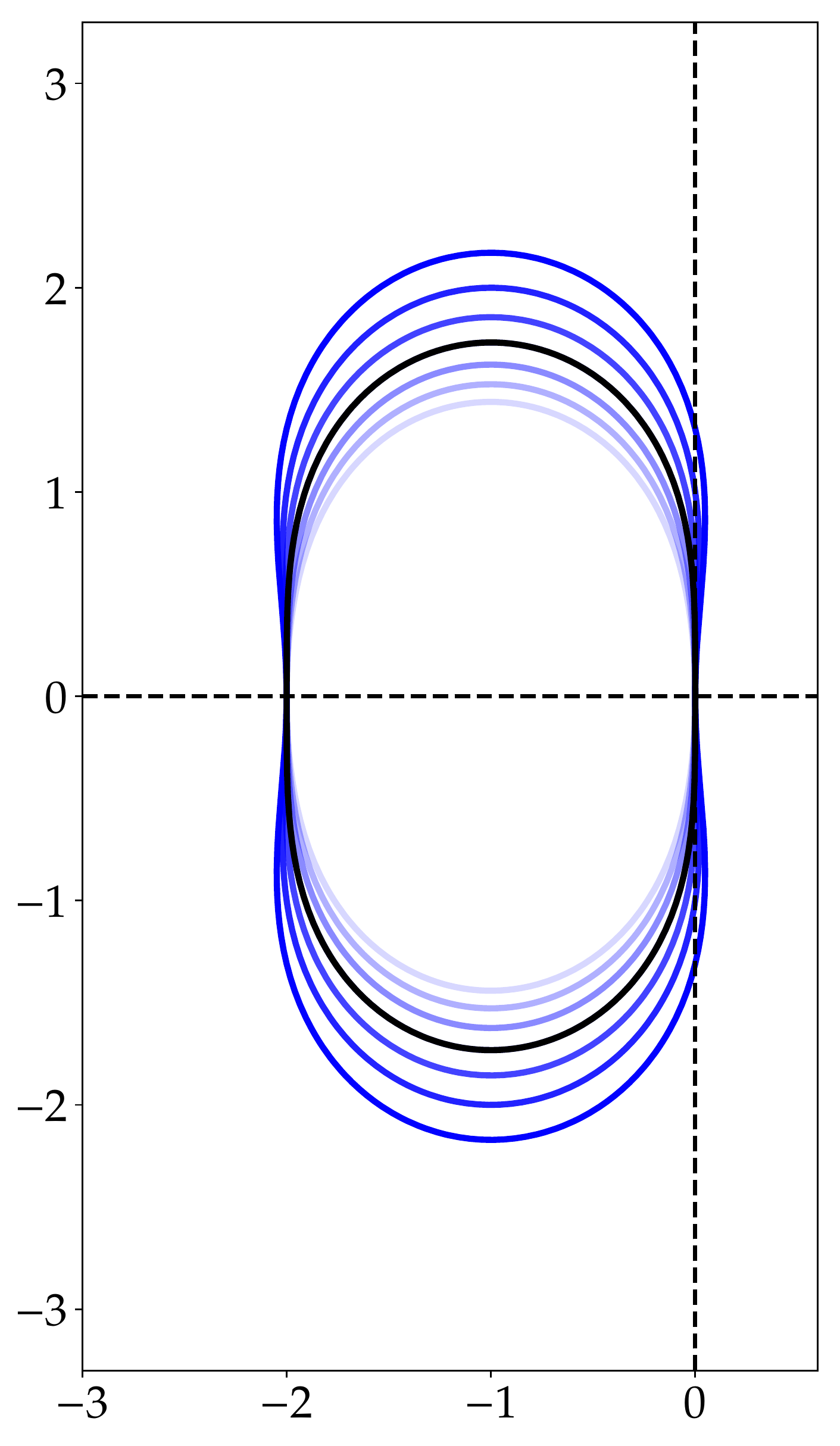}
        \caption{2-stage, 2nd-order RRK methods.
        \label{fig:rk2stabreg}}
    \end{subfigure}~
    \begin{subfigure}[t]{1.6in}
        \includegraphics[width=1.6in]{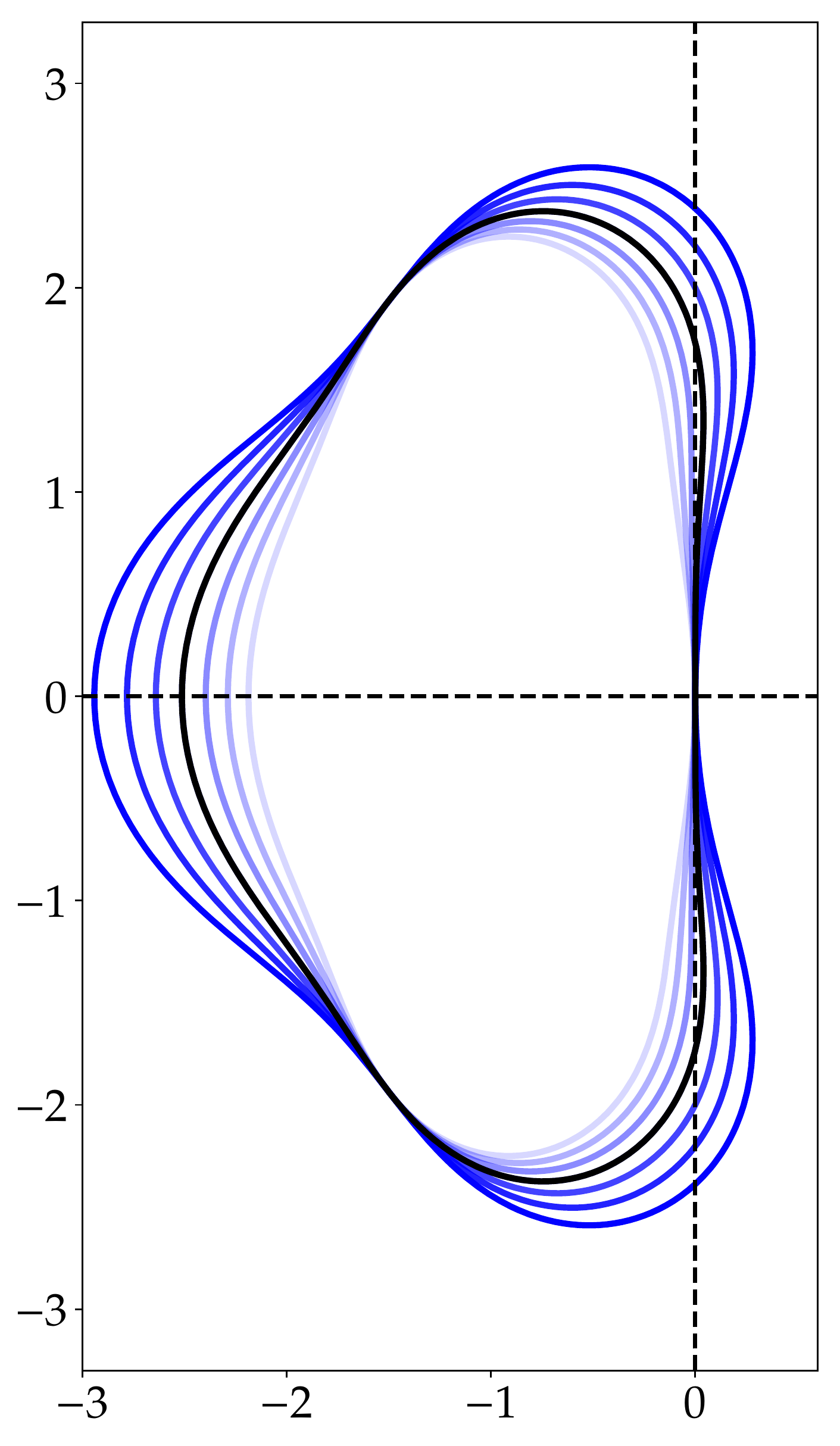}
        \caption{3-stage, 3rd-order RRK methods.
        \label{fig:rk3stabreg}}
    \end{subfigure}~
    \begin{subfigure}[t]{1.6in}
        \includegraphics[width=1.6in]{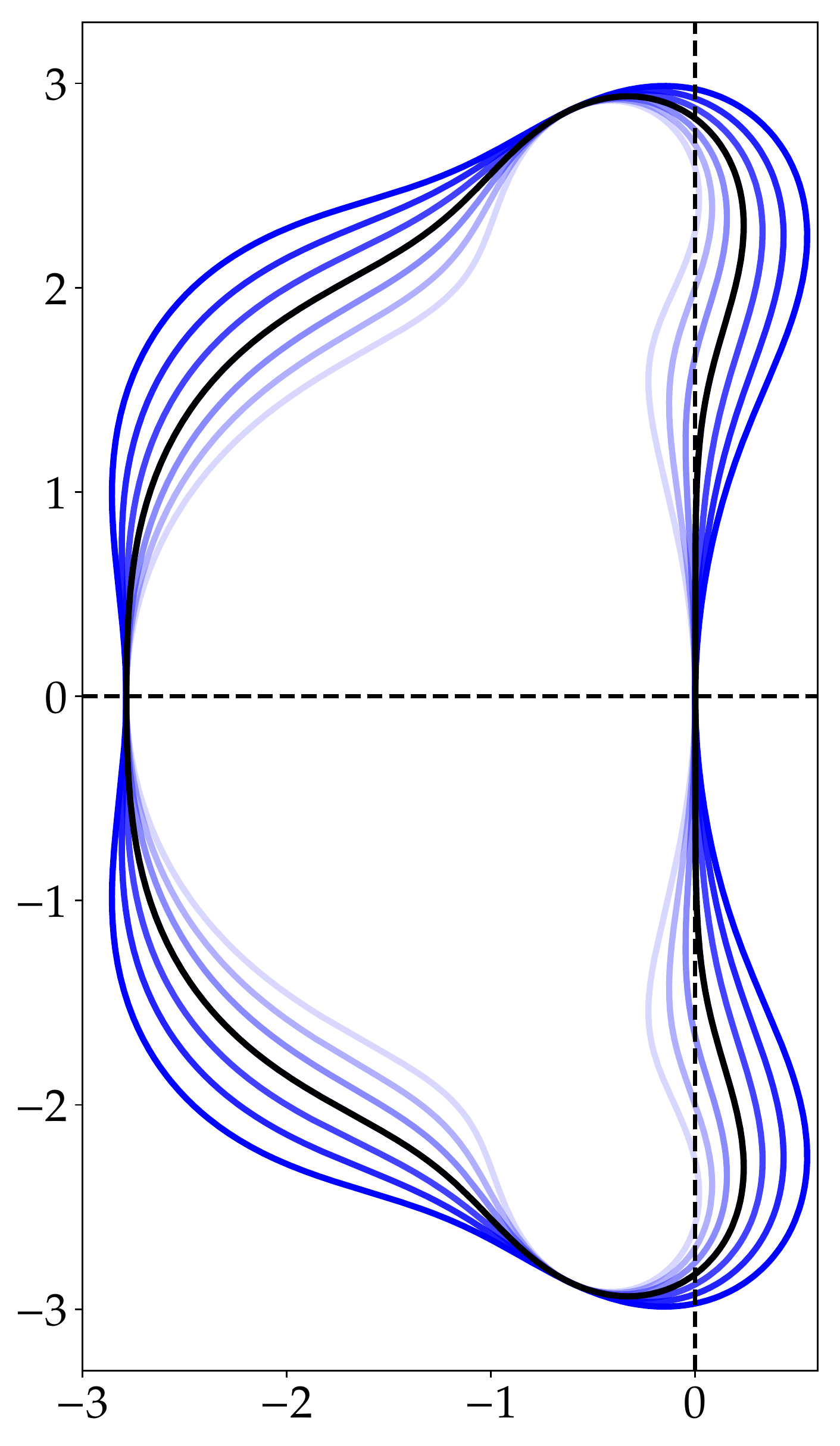}
        \caption{4-stage, 4th-order RRK methods.
        \label{fig:rk4stabreg}}
    \end{subfigure}
    \caption{Stability regions for RRK methods.  Regions are shown for $\gamma$ ranging
            from $0.7$ to $1.3$.  Larger regions correspond to smaller values of $\gamma$.
            The stability region of the original RK method (corresponding to $\gamma=1$)
            is outlined in black.
            \label{fig:stabregs}}
\end{figure}

Consideration of absolute stability along the imaginary axis leads to the
so-called E-polynomial, which for explicit RK methods is
$$
E(y) = 1 - |R(iy)|^2 = 1 - R(iy)R(-iy).
$$
Clearly the method is stable for $z=iy$ such that $E(y) \ge 0$.  Direct calculation shows that
for any RRK method of order at least $p$ (where $p\ge 2$), we have
$$
E_\gamma(y) = \sum_{j=1}^{p/2} \frac{2}{(2j)!} \gamma(1-\gamma)y^{2j} + \bigoh(y^{p+2}).
$$
The leading terms up to $y^p$ (which vanish for a standard RK method) are
positive for $0\le \gamma < 1$, so we see that for
every method of order two or higher $S(A,\gamma b)$ includes a segment of the imaginary
axis containing $z=0$ when $\gamma\in[0,1)$.  This can be observed for instance in Figure \ref{fig:rk2stabreg},
where the RK method is unstable over the whole imaginary axis, but for $\gamma <1$ the
stability regions include part of the imaginary axis.

\subsection{Strong stability preservation}
Because the RRK method is only a small perturbation of the original RK method,
desirable properties of the original method may remain in effect when using the
RRK version.  We illustrate this idea by studying strong stability preserving
RRK methods.

In the following, $\C(A,b)$ denotes the SSP coefficient (or radius of absolute
monotonicity) of the Runge--Kutta method with coefficients $(A,b)$.
Recall that $\C(A,b)$ is equal to the largest value $r\ge0$ such that the method
$(A,b)$ is absolutely monotonic at $-r$ \cite{ferracina2005}.

For $0\le\gamma\le 1$, $u_\gamma^{n+1}$ is a convex combination of $u^n$ and
$u^{n+1}$, which implies that $\C(A,\gamma b)\ge \C(A,b)$ for $0\le \gamma \le 1$.
For many methods, $\C(A,\gamma b)$ also does not decrease when $\gamma$
is taken a little larger than 1.

\begin{lemma} \label{SSPlemma}
Let the RK method with coefficients $(A,b)$ be absolutely monotonic
at $z=-r$.  Then the method $(A,\gamma b)$ with $\gamma\ge 0$ is also
absolutely monotonic at $z=-r$ iff $R_\gamma(-r)\ge 0$.
\end{lemma}
\begin{proof}
The conditions for absolute monotonicity of method $(A,b)$
at $z$ can be written (see, e.g. \cite[p. 211]{ferracina2005}
\begin{subequations} \label{absmoncond}
\begin{align}
A(I-zA)^{-1} & \ge 0 & (I-zA)^{-1}e \ge 0 \label{Aconds} \\
b^T(I-zA)^{-1} & \ge 0 & R(z) \ge 0. \label{bconds}
\end{align}
\end{subequations}
Conditions \eqref{Aconds} do not depend on
$b$, while the first condition of \eqref{bconds} will hold for any positive multiple of
$b$ if it holds for $b$.
\end{proof}

\begin{theorem} \label{SSPthm}
Given a RK method with coefficients $(A,b)$, we have $\C(A,\gamma b) \ge \C(A,b)$
for $0 \le \gamma \le \gamma_*$, where
$$
    \gamma_* = -\frac{1}{\sum_k \alpha_k(-\C(A,b))^k} = \frac{-1}{R(-\C)-1} \ge 1.
$$
\end{theorem}
\begin{proof}
    Combining Lemma \ref{SSPlemma} with a standard result on absolute monotonicity
    (see e.g. \cite[Lemma 3.1]{kraaijevanger1991}) we have that $(A,\gamma b)$ is absolutely
    monotonic on $[-\C(A,b), 0]$ as long as $R_\gamma(-\C(A,b))\ge 0$.
    Since $R_\gamma(z) - 1 = \gamma(R(z)-1)$, we obtain the condition stated in the theorem.
\end{proof}
Values of $\gamma_*$ are given in Table \ref{tbl:mu} for some well-known SSP methods.
For all of these methods, direct computation shows that $\C(A,\gamma b) = \C(A,b)$ for $0\le \gamma \le \gamma_*$,
while taking $\gamma > \gamma_*$ leads to a decrease in the SSP coefficient.

\begin{table}
\center
\begin{tabular}{cc}
Method & $\gamma_*$ \\ \hline
SSPRK(2,2) & 2 \\
SSPRK(s,2) & s/(s-1) \\
SSPRK(3,3) & 3/2 \\
SSPRK(4,3) & 1 \\
SSPRK(5,3) & 1 \\
SSPRK(9,3) & 1 \\
SSPRK(5,4) & 1.312 \\
SSPRK(10,4) & 25/24 \\ \hline
\end{tabular}
\caption{Values of $\gamma_*$ from Theorem \ref{SSPthm} for some well-known SSP methods.
For $0\le \gamma \le \gamma_*$, the relaxation method has the same SSP coefficient as the
original method.
\label{tbl:mu}}
\end{table}

\section{Numerical examples} \label{examples}
For conservative systems, stability is guaranteed under
any step size.  However, one expects the accuracy to deteriorate significantly if $\gamma$ is not
close to unity.  In the dissipative case, it is possible that $\gamma$ becomes negative and
stability is lost.  A careful examination of the analysis in the previous sections
suggests that step sizes on the order of what linearized stability analysis predicts
to be stable should be acceptable.

The following Runge--Kutta methods will be used in the numerical experiments.
\begin{itemize}
  \item
  SSPRK(2,2): Two-stage, second-order SSP method of \cite{shu1988}.

  \item
  SSPRK(3,3): Three-stage, third-order SSP method of \cite{shu1988}.

  \item
  SSPRK(10,4): Ten-stage, fourth-order SSP method of \cite{ketcheson2008highly}.

  \item
  RK(4,4): Classical four-stage, fourth-order method.

  \item
  BSRK(8,5): Eight-stage, fifth-order method of \cite{bogacki1996efficient}.
  A fixed step size $\Dt$ is used and the embedded method is not used.
\end{itemize}
All of these methods have non-negative weights.  We refer to the
relaxation version of a method by replacing ``RK'' with ``RRK''; e.g. RRK(4,4).

\subsection{Linear, skew-Hermitian system}
With this example we investigate the behavior of RRK methods for linear skew-Hermitian
problems.  For concreteness, we consider the advection equation
\begin{align}
    U_t & = U_x & U(x,0) = U^0(x)
\end{align}
with periodic boundary conditions over a spatial interval of length $2\pi$
discretized in space by a Fourier spectral collocation method
with $m$ points.  This results in a linear, constant-coefficient system of ODEs:
$$
    u'(t) = Du(t)
$$
where $u_j(t) \approx U(x_j,t)$ at evenly spaced points $x_j$ and $D$ is the
$m\times m$ skew-Hermitian Fourier spectral differentiation matrix.
Since $D$ is normal, the behavior of any RK method on this problem
can be characterized simply in terms of the eigenvalues of $D$ and the corresponding
eigenvectors, which are discrete Fourier modes:
\begin{align*}
    \lambda_\xi & = i \xi & \xi = -\frac{m}{2}, -\frac{m}{2}+1, \dots,  \frac{m}{2}-1 \\
    v_\xi & = [\exp(i \xi x_1), \exp(i \xi x_2), \dots, \exp(i\xi x_m)]^T.
\end{align*}
Let us express the initial data in terms of these modes:
$$
u^0 = \sum_\xi \hat{u}^0_\xi v_\xi,
$$
where $\hat{u}$ is the discrete Fourier transform (DFT) of $u$.
Then the exact solution of the semi-discrete system is given by
$$
u(t) = \sum_\xi e^{i \xi t} \hat{u}^0_\xi v_\xi.
$$
Thus the energy associated with each mode is constant in time.
Applying a Runge--Kutta method we instead obtain the solution
$$
u^n = \sum_\xi R(i \xi \Dt)^n \hat{u}^0_\xi v_\xi.
$$
where $R(z)$ is again the stability function of the method.
The energy associated with mode $\xi$ is modified
by the factor $|R(i\xi\Dt)|$ at each step.  The maximum stable step size is
the value that guarantees $i \xi \Dt \in S(A,b)$ for all $\xi$.
This is just
$$
    \Dt_{\max} = \frac{I(A,b)}{\max_\xi|\lambda_\xi|} = \frac{2}{m} I(A,b),
$$
where $I(A,b)$ is the length of the method's imaginary axis stability interval.
In the following experiments, we use a step size $\Dt = \mu \Dt_{\max}$.
Using a given (standard) RK method and $0 \le \mu \le 1$, we have absolute stability and
the energy of each mode decays.  This is illustrated in Figure \ref{fig:rk4_energydecay},
where we solve with $m=128$ and plot the relative change in amplitude
$$
\frac{|\hat{u}^n_\xi| - |\hat{u}^0_\xi|}{|\hat{u}^0_\xi|}
$$
for each mode, for a range of time step sizes $\mu \Dt_{\max}$.  Due to symmetry we plot only
the positive wavenumbers.  This figure does not depend on the initial data
but only on $|R(iy)|^N$, where $N$ is the total number of time steps taken.
For larger values of $\mu$, the high wavenumber modes are strongly damped.

Now let us consider what happens when applying a relaxation Runge--Kutta
method with $\gamma_n$ chosen according to \eqref{gammadefreal} so that energy is
conserved.  At each step,
the energy in mode $\xi$ is modified by the factor $|R_\gamma(i \xi \Dt)|$.
If the initial data is chosen to consist of a single wavenumber $k_\xi$, then $\gamma_n$
will take a value such that $|R_\gamma(i\xi \Dt)|=1$, and the same value of
$\gamma_n$ will be used at every step.  For more general initial data,
$\gamma_n$ is chosen precisely so that the change
in energy when summed over all modes is zero.  This value depends on the data,
so $\gamma_n$ will be different at each step.  Furthermore, at each step some
modes will be diminished while others will grow.
This is illustrated in
Figure \ref{fig:rrk4_energydecay}, which is analogous to Figure \ref{fig:rk4_energydecay}
but for the energy-conserving RRK(4,4) method, with initial data taken as white noise;
i.e. $\hat{u}^0_j = e^{i\theta_j}$ where the phases $\theta_j$ are random.
We see that high-wavenumber modes
are again damped, especially for step sizes close to the stability limit.
Meanwhile, some lower-wavenumber modes are amplified in order to preserve the
total energy.
If we instead take initial data that is reasonably well-resolved, such as
\begin{align} \label{sech2IC}
    U^0(x) = \sech^2(7.5(x+1)).
\end{align}
the resulting amplification curves are nearly indistinguishable from
those of the standard RK4 method (Figure \ref{fig:rk4_energydecay}), as shown
in Figure \ref{fig:rrk4_energydecay2}.
This is because most of the energy is in the low-wavenumber modes, which
are propagated fairly accurately by the standard RK4 method, so little
compensation is need in order to conserve energy.
In the latter figure we also plot the amplification for $\mu=1.02$, which is
beyond the absolute stability limit.  We see that in this case the RRK method
greatly amplifies the highest-wavenumber mode.

\begin{figure}
    \centering
    \begin{subfigure}{\textwidth}
      \centering
      \includegraphics[width=\textwidth]{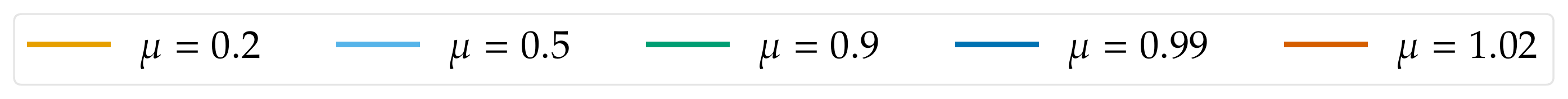}
    \end{subfigure}
    \\
    \begin{subfigure}[t]{0.32\textwidth}
        \centering\captionsetup{width=.8\linewidth}
        \includegraphics[width=\textwidth]{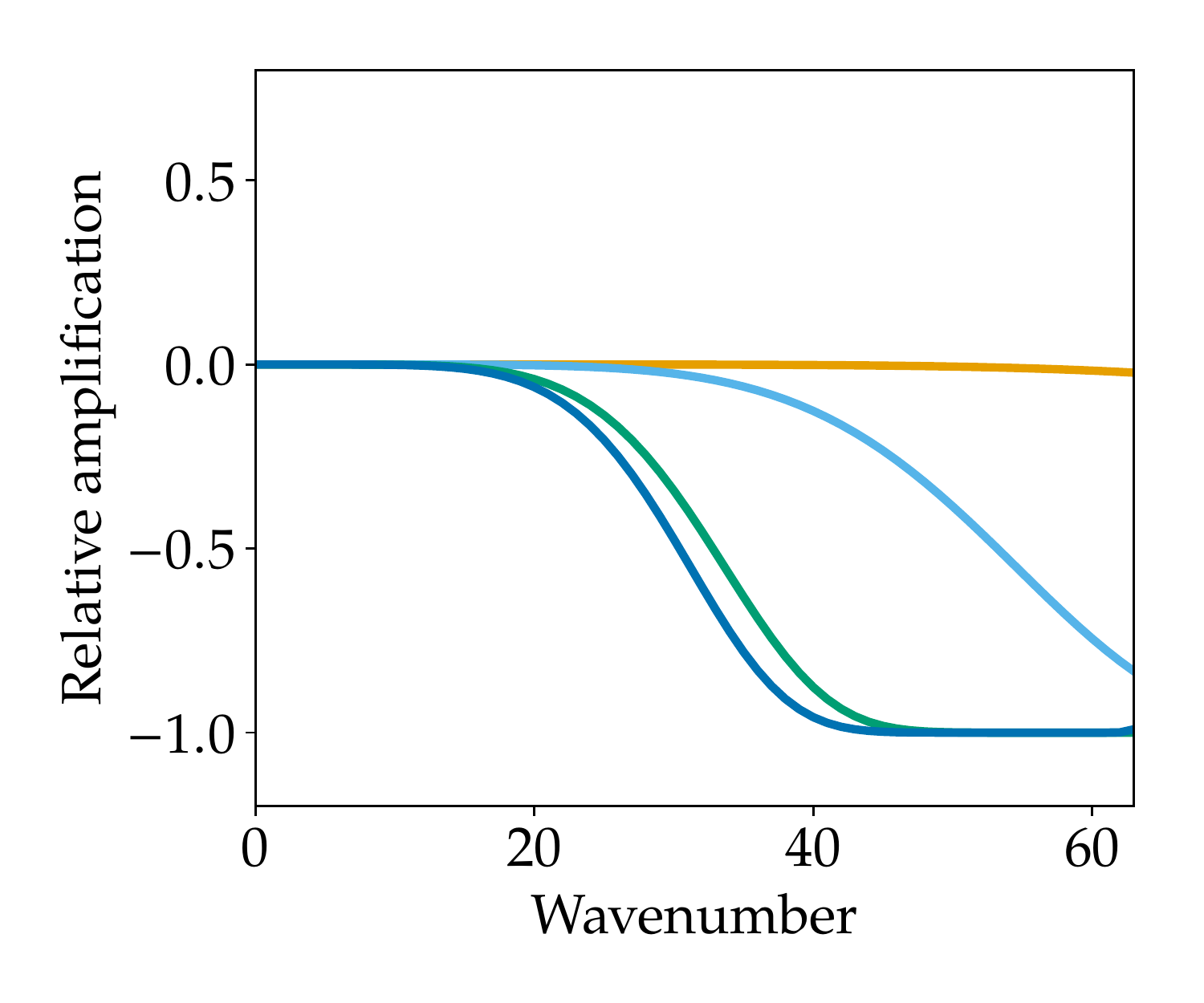}
        \caption{Standard RK(4,4) method
        \label{fig:rk4_energydecay}}
    \end{subfigure}
    \begin{subfigure}[t]{0.32\textwidth}
        \centering\captionsetup{width=.8\linewidth}
        \includegraphics[width=\textwidth]{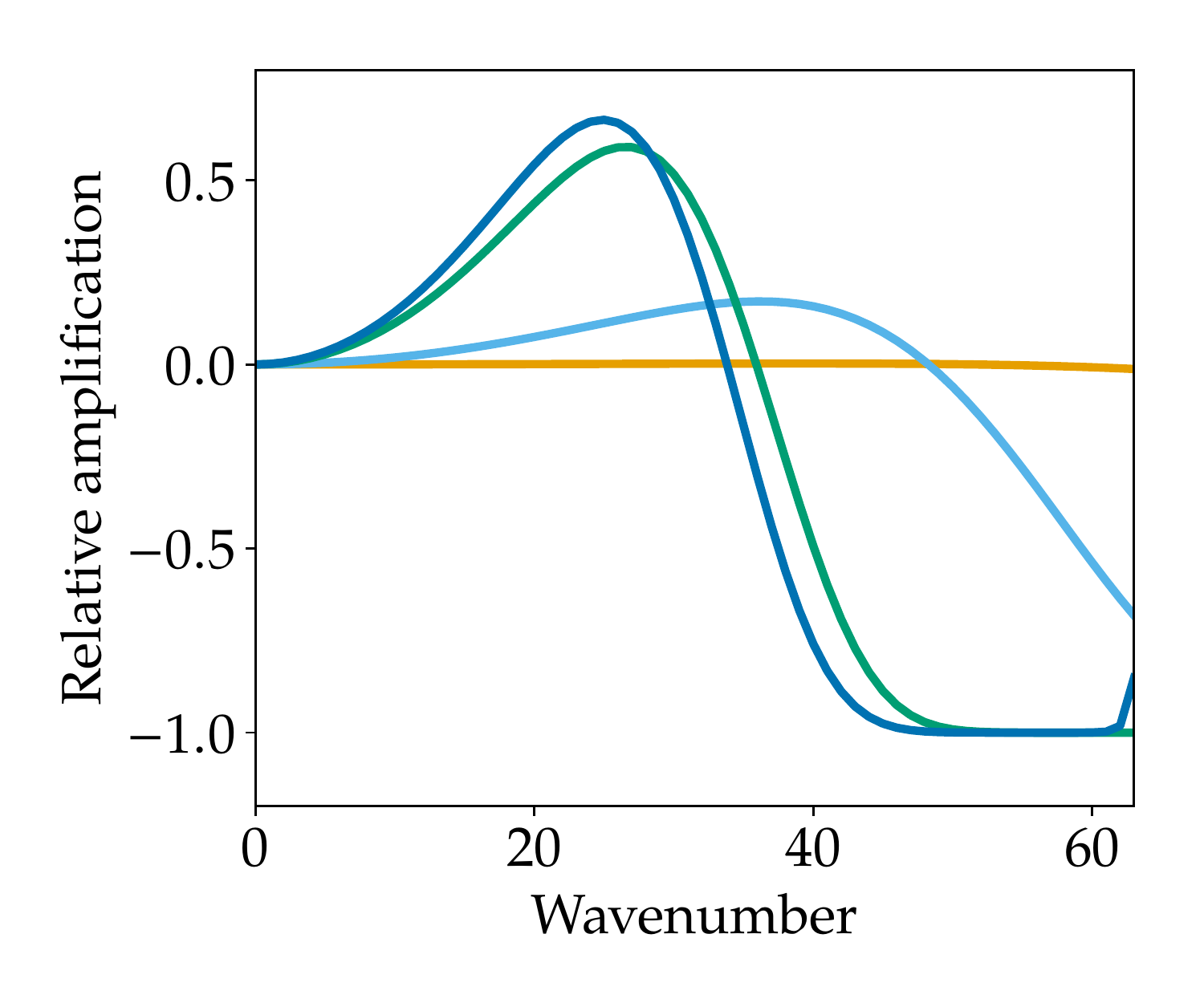}
    \caption{Energy-conserving RRK(4,4) method, white noise initial data
    \label{fig:rrk4_energydecay}}
    \end{subfigure}
    \begin{subfigure}[t]{0.32\textwidth}
        \centering\captionsetup{width=.8\linewidth}
        \includegraphics[width=\textwidth]{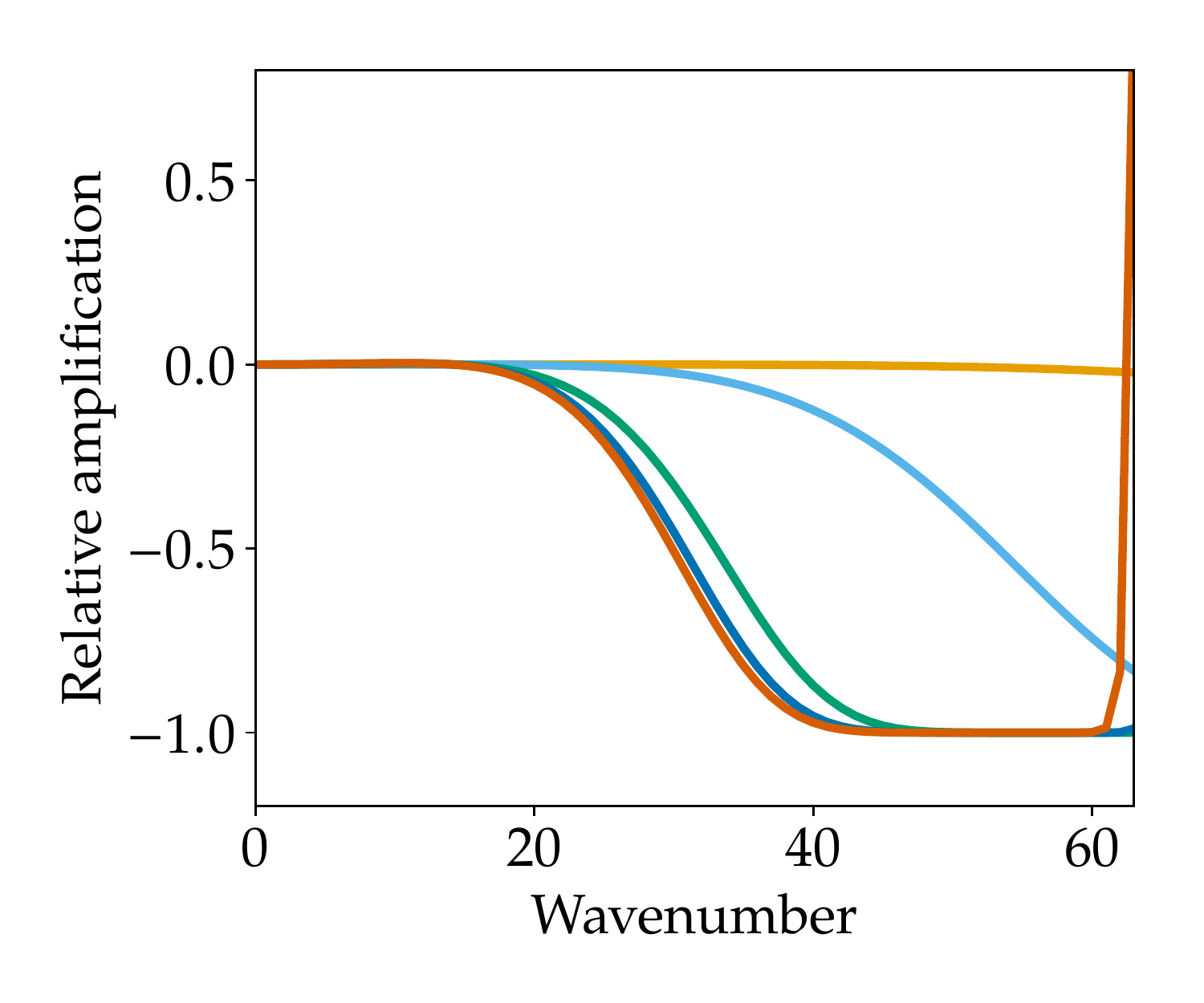}
    \caption{Energy-conserving RRK(4,4) method; smooth initial data \eqref{sech2IC}
    \label{fig:rrk4_energydecay2}}
    \end{subfigure}
    \caption{Relative amplification of each mode for the spectral semi-discretization
    of the advection equation, integrated up to $t=1$, for step sizes $\Dt = \mu\Dt_{\max}$.
    }
\end{figure}

In Figure \ref{fig:waveeq_comparison}, we compare the behavior of RK(4,4) and relaxation RK(4,4) (RRK(4,4))
for different values of $\mu$.  For $\mu < 1$, the methods give similar solutions.
Unlike RK(4,4), RRK(4,4) remains stable for $\mu>1$.  However,
taking $\mu=1.016$ leads to highly inaccurate and oscillatory approximations.

\begin{remark}
    Even for the linearly unstable value $\mu=1.016$, we have found that $\gamma$ remains
    within less than $10^{-2}$ of unity and the solution never blows up.
    For much larger values ($\mu \gtrsim 1.25$)
    the value of $\gamma$ tends to zero after a few steps and the calculation is never
    completed.
\end{remark}

\begin{figure*}
    \centering
    \begin{subfigure}[t]{0.5\textwidth}
        \centering
        \includegraphics[width=\textwidth]{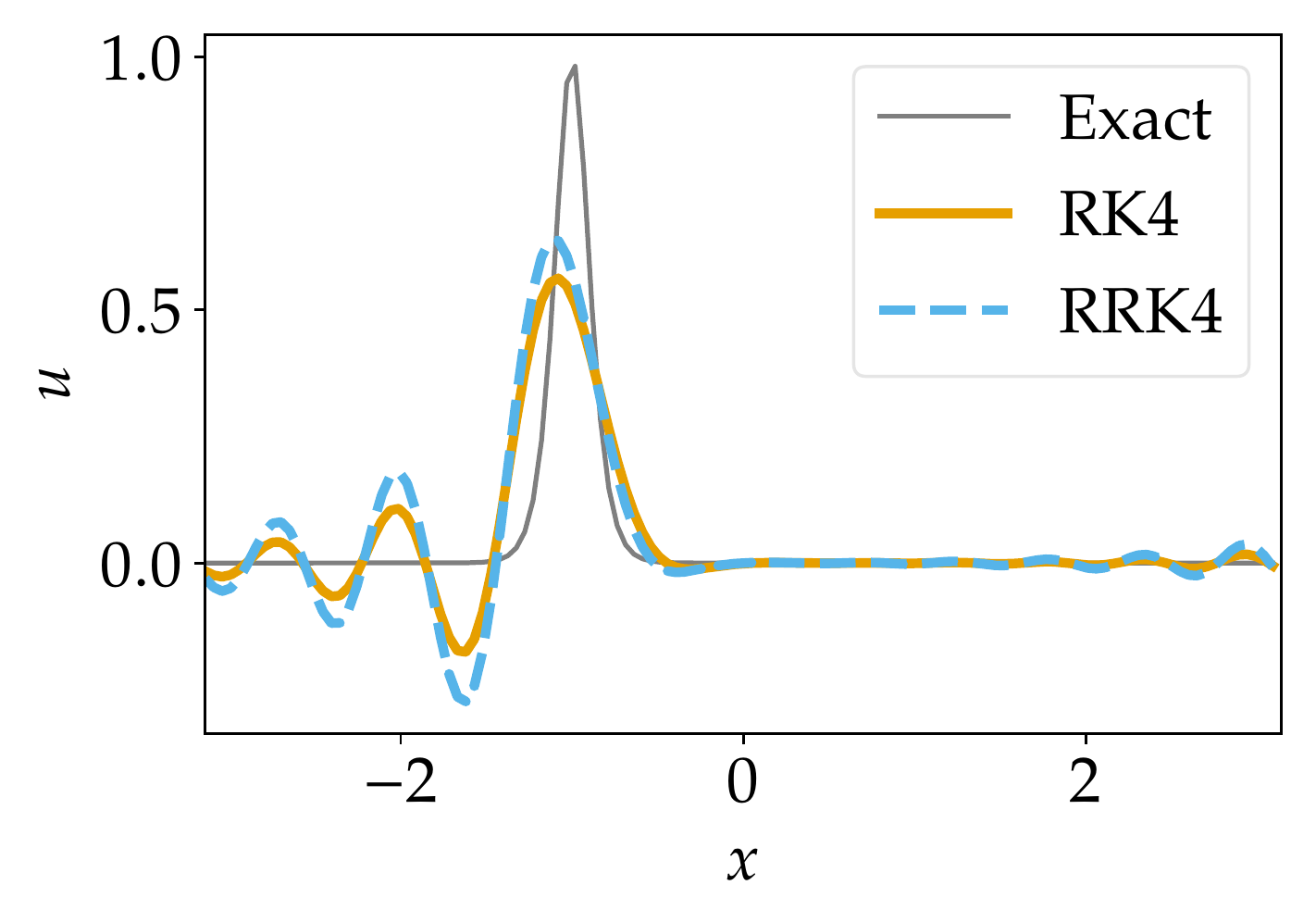}
        \caption{RK(4,4) vs. RRK(4,4); $\mu=0.99$.}
    \end{subfigure}%
    \begin{subfigure}[t]{0.5\textwidth}
        \centering
        \includegraphics[width=\textwidth]{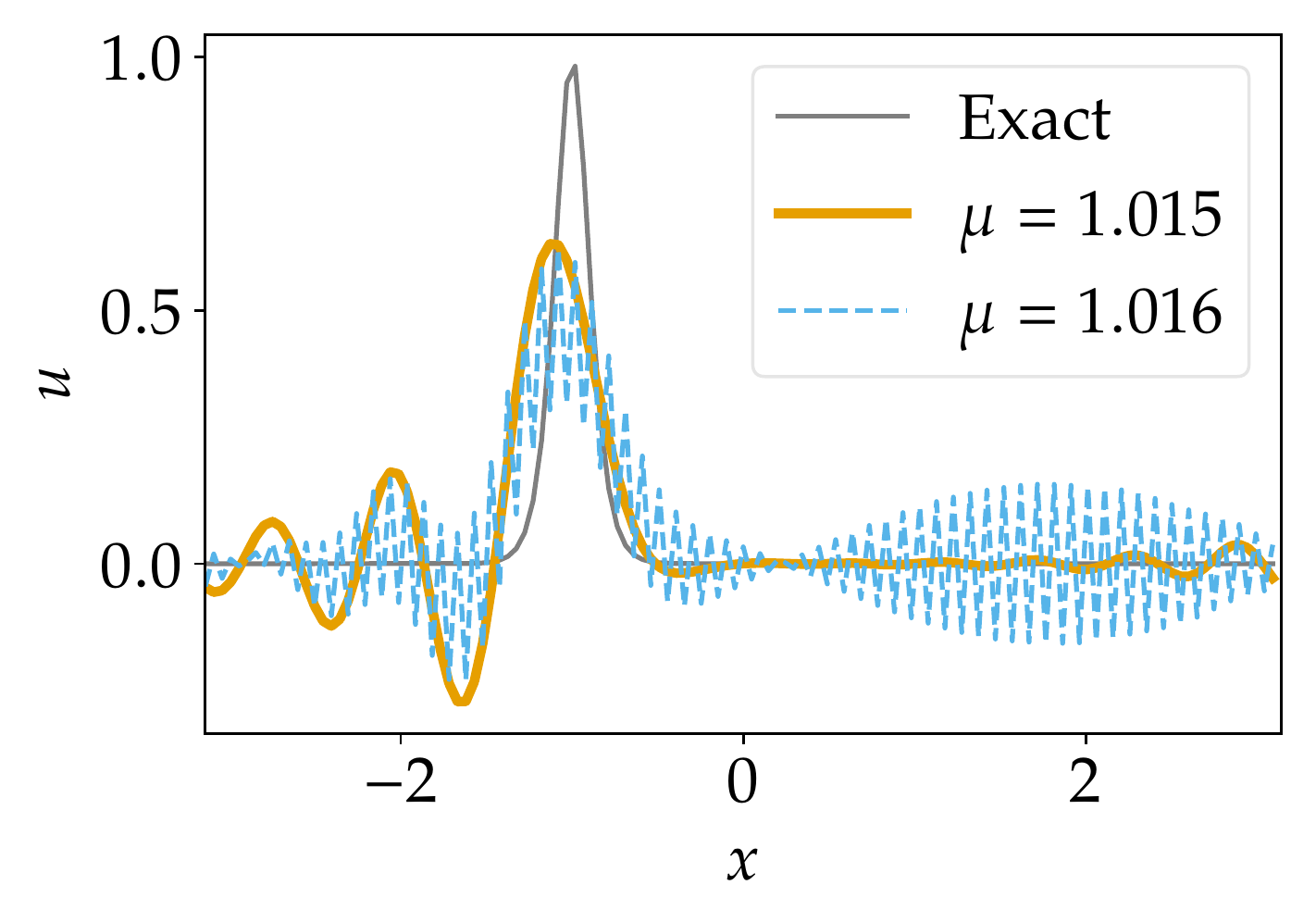}
        \caption{RRK(4,4) with two slightly different step sizes}
    \end{subfigure}
    \caption{Solutions to the advection equation at $t=400\pi$, using standard RK(4,4)
             and the energy-conserving relaxation modification of RK(4,4).
             The absolute stability limit corresponds to $\mu=1$.
             The two methods give very similar solutions for all stable
             step sizes.  The RK(4,4) solution blows up for $\mu>1$; the RRK(4,4)
             solution is stable for all step sizes but becomes highly oscillatory
             for $\mu > 1.015$.
             \label{fig:waveeq_comparison}}
\end{figure*}

\subsection{A non-normal linear autonomous problem}
Here we consider a linear autonomous problem $u'(t) = Au(t)$
with non-normal right-hand side introduced by Sun \& Shu \cite{sun2017}:
\begin{align}
    \begin{bmatrix} u_1 \\ u_2 \\ u_3 \end{bmatrix}' & =
    \begin{bmatrix} -1 & -2 & -2 \\ 0 & -1 & -2 \\ 0 & 0 & -1 \end{bmatrix}
    \begin{bmatrix} u_1 \\ u_2 \\ u_3 \end{bmatrix}.
\end{align}
This problem is dissipative, but --- as shown in \cite{sun2017} ---
the classical 4th-order Runge--Kutta method RK(4,4) is not monotone for this
problem, no matter how small one takes $\Dt$.  To provide a concrete example,
we compute  $R(\Dt A)$, where $R(z)$ is
the stability polynomial of RK(4,4) and we choose step size $\Dt=0.5$.
The first singular value of the matrix $R(\Dt A)$ is approximately $1.001$.
Taking a single step with RK(4,4) and initial condition equal to the first right
singular vector thus leads to an
increase in the energy, as shown in Figure~\ref{fig:sunshu}.
The increase is even larger for the step size 0.7.
In contrast, RRK(4,4) preserves monotonicity with either step size.

\begin{figure}
\center
\includegraphics[width=3in]{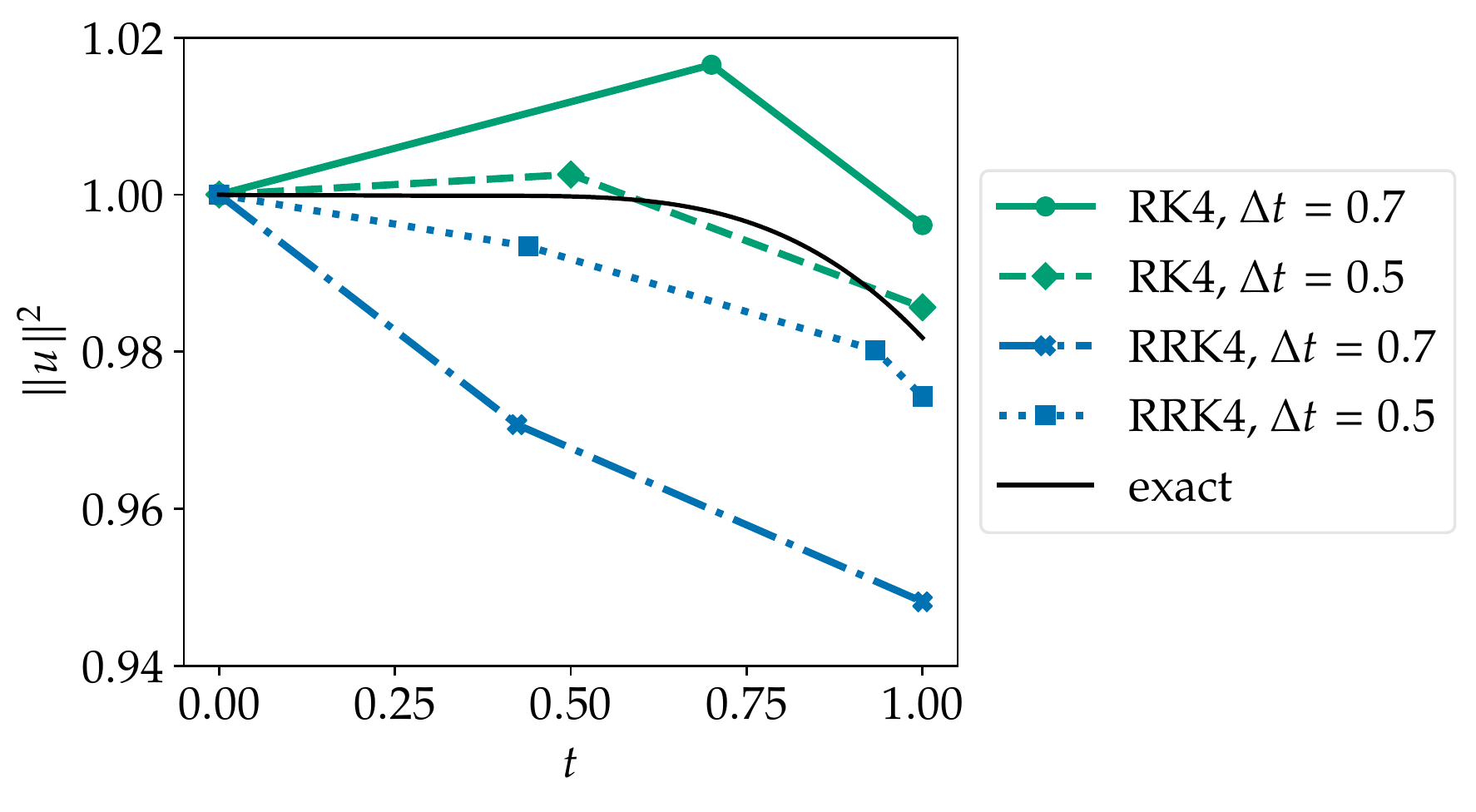}
\caption{Standard RK(4,4) and its energy-conserving modification RRK(4,4) applied to
the test problem of Sun \& Shu\cite{sun2017}.
The RK(4,4) solution gives an increased energy at the first step, while the RRK(4,4)
solution is monotone.
\label{fig:sunshu}}
\end{figure}

\subsection{Nonlinear oscillator}
Here we consider the problem
\begin{equation} \label{prob1}
  \begin{bmatrix} u_1 \\ u_2 \end{bmatrix}'
  =
  \frac{1}{\| u \|^2}
  \begin{bmatrix} -u_2 \\ u_1 \end{bmatrix},
  \quad
  \begin{bmatrix} u_1(0) \\ u_2(0) \end{bmatrix}
  =
  \begin{bmatrix} 1 \\ 0 \end{bmatrix},
\end{equation}
with analytical solution
\begin{equation*}
  \begin{bmatrix} u_1(t) \\ u_2(t) \end{bmatrix}
  =
  \begin{bmatrix} \cos(t) \\ \sin(t) \end{bmatrix}.
\end{equation*}
Although energy is conserved in the exact solution, the widely used
SSPRK(3,3) method of Shu \& Osher \cite{shu1988} produces a solution whose energy is
monotonically increasing for every positive
step size \cite{ranocha2018strong}.  Similar behavior is observed for several other explicit RK
methods we have tested, as shown in Figure \ref{fig:prob1_energy_RK}.
By applying the modification described in the present work, we obtain instead
Figure \ref{fig:prob1_energy_RRK}, showing that energy is conserved up to roundoff
error for all methods.
\begin{figure}
    \centering
    \begin{subfigure}{\textwidth}
      \centering
      \includegraphics[width=\textwidth]{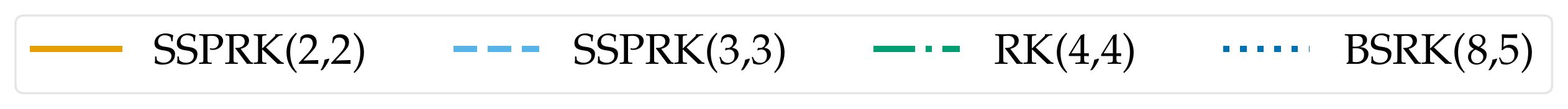}
    \end{subfigure}
    \\
    \begin{subfigure}[b]{0.49\textwidth}
        \includegraphics[width=\textwidth]{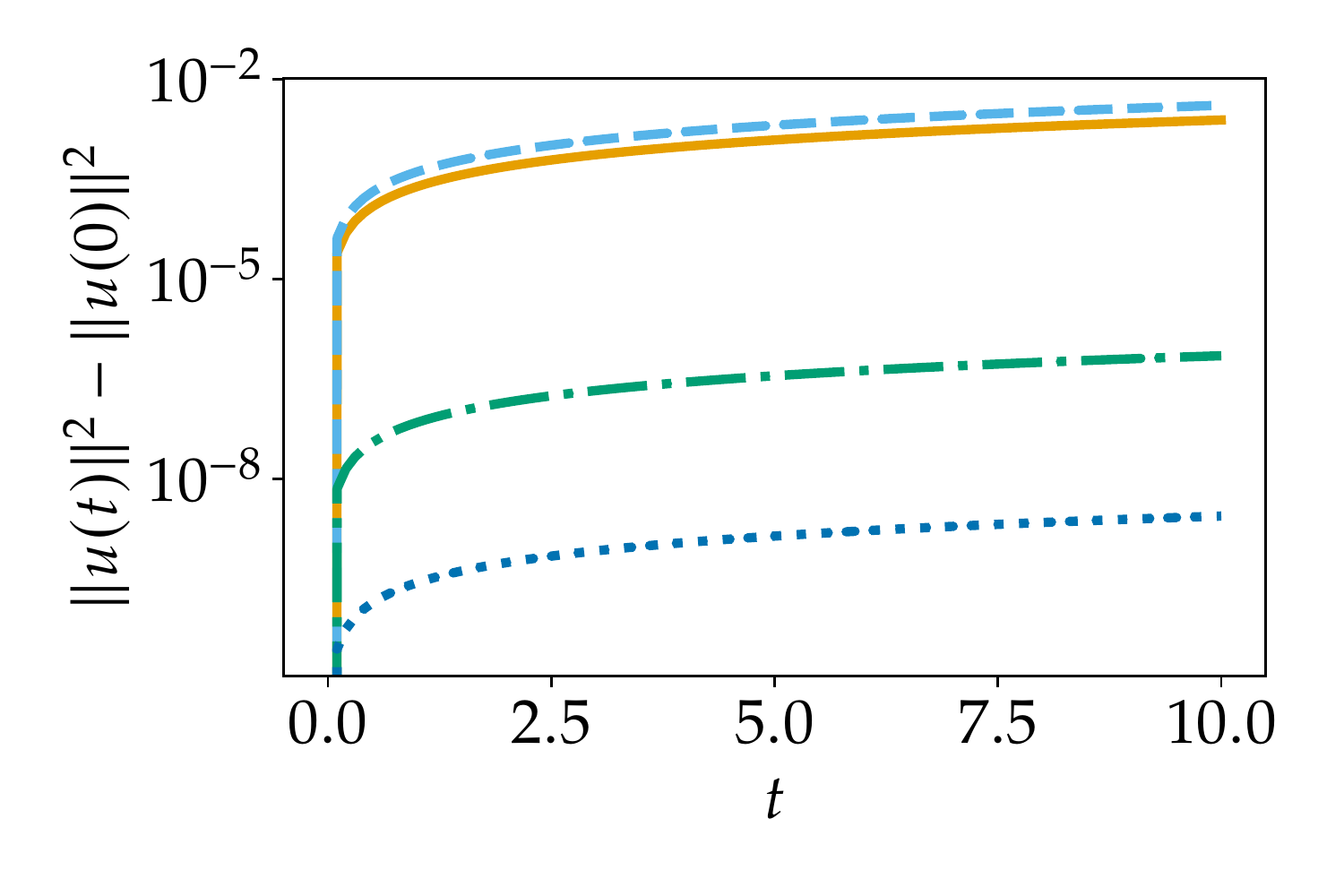}
        \caption{Standard RK methods\label{fig:prob1_energy_RK}}
    \end{subfigure}~
    \begin{subfigure}[b]{0.49\textwidth}
        \includegraphics[width=\textwidth]{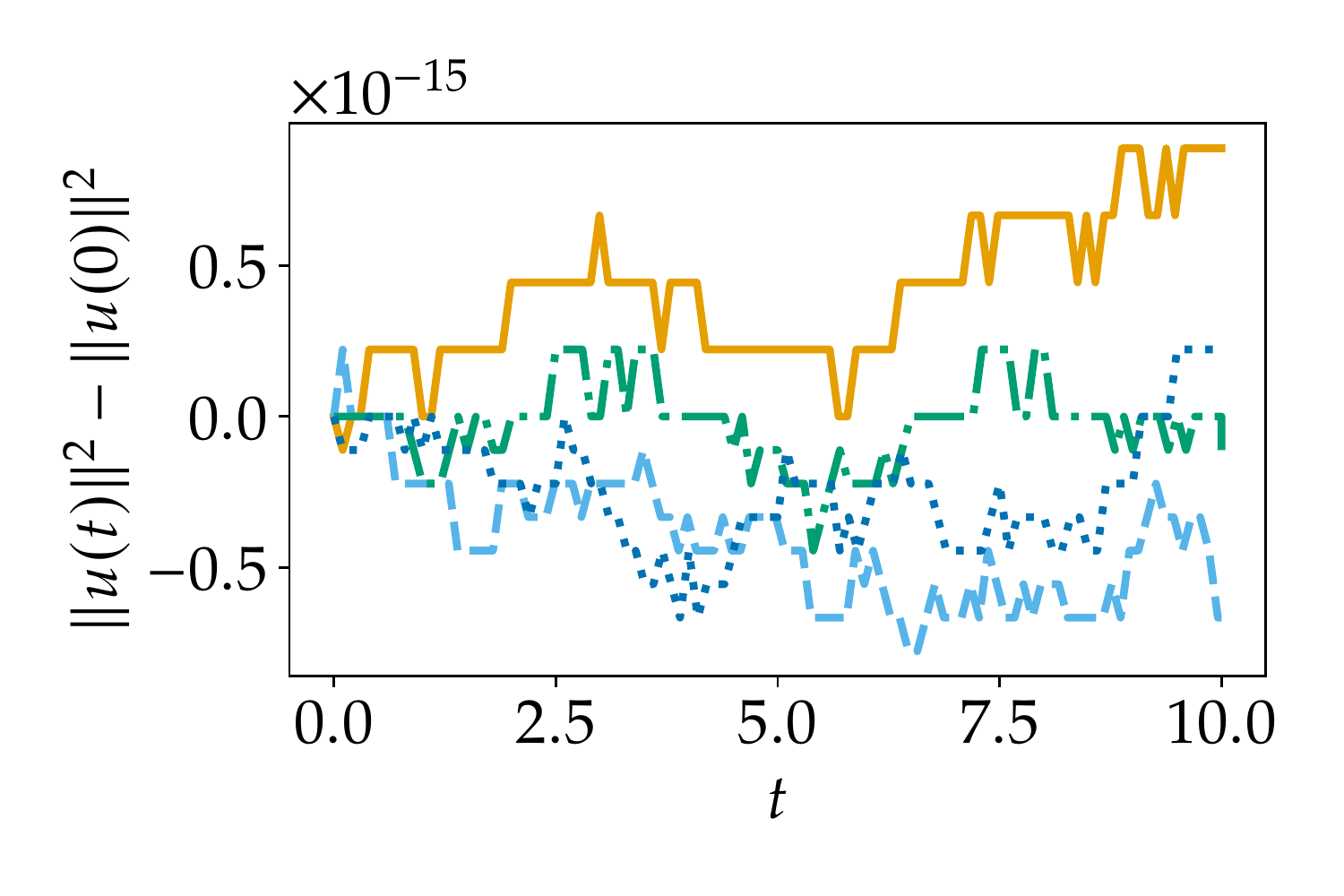}
    \caption{Relaxation RK methods\label{fig:prob1_energy_RRK}}
    \end{subfigure}
    \caption{Evolution of the energy in the numerical solution of \eqref{prob1}.
            Standard RK methods exhibit energy growth or decay, while relaxation
            methods conserve energy to within roundoff error.  Note that the scale
            of the left figure is logarithmic, while the scale of the right figure
            is linear.
    }
\end{figure}

Figure \ref{fig:prob1_convergence} shows the convergence for each of the standard RK methods
(solid lines) and its energy-conserving modification (dashed lines).  We see that
the relaxation method is in each case on par with or more accurate than the
original method.

\begin{figure}
\center
\includegraphics[width=2.5in]{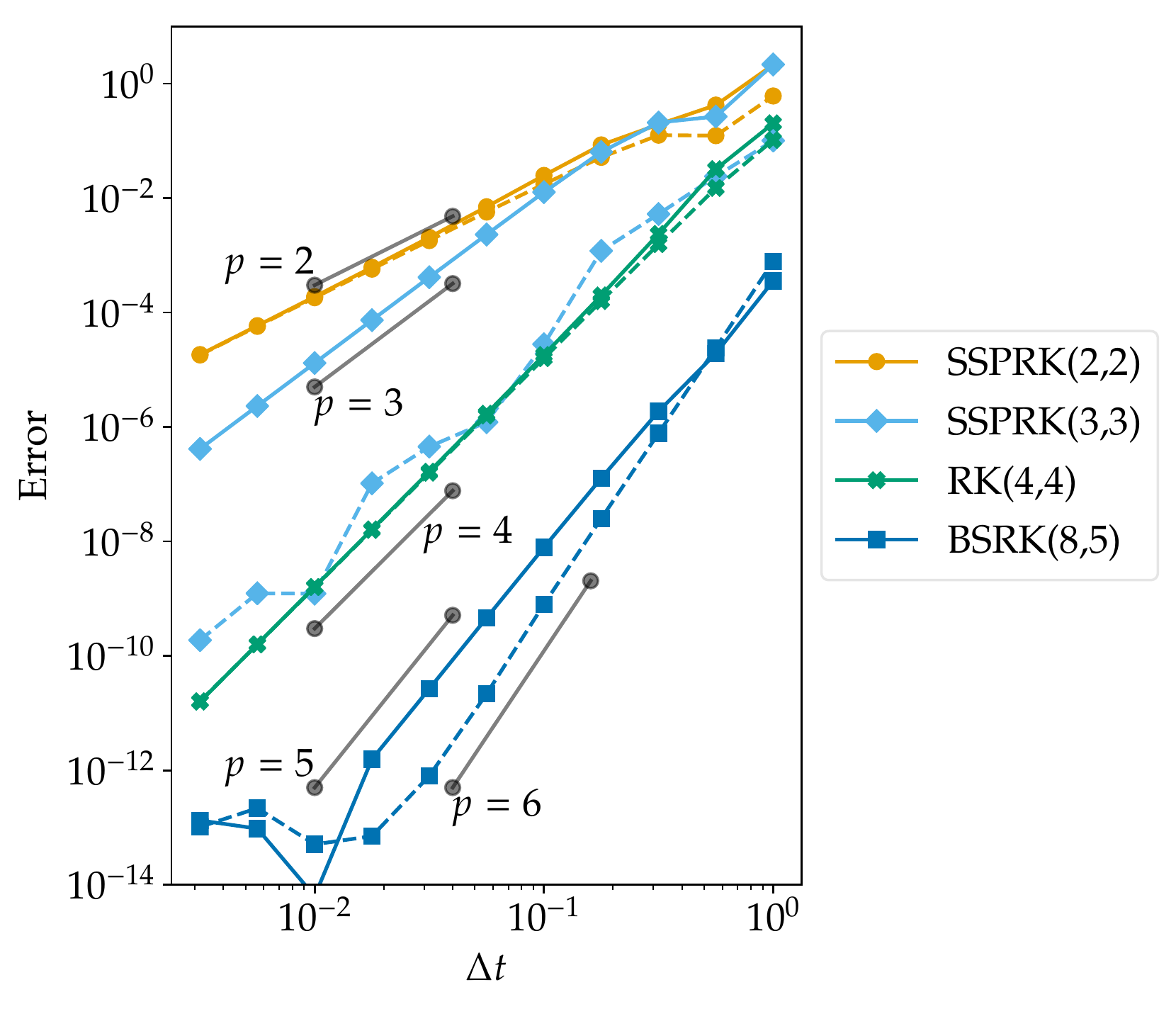}
\caption{Convergence of several methods and their energy-conserving relaxation
variants on the test problem \eqref{prob1}.  Solid lines are standard RK methods;
dashed lines are relaxation RK methods.
\label{fig:prob1_convergence}}
\end{figure}

\subsection{Burgers' equation}
We solve the inviscid Burgers' equation
\begin{align} \label{burgers}
    U_t  + \frac{1}{2}(U^2)_x & = 0
\end{align}
on the interval $-1 \le x \le 1$ with periodic boundary conditions and
initial data
$$
U(x,0) = \exp(-30x^2)
$$
with the flux-differencing discretization
$$
    u_i'(t) = - \frac{1}{\Delta x} (F_{i+1/2} - F_{i-1/2})
$$
where $F_{i\pm 1/2}$ is the numerical flux, defined below.
The spatial domain is discretized with 50 equally-spaced points.
In the convergence tests below, this spatial discretization
is held fixed while the time step is varied, in order to
investigate only the temporal convergence.

\subsubsection{Energy-conservative semi-discretization}
We take the second-order accurate symmetric flux \cite{tadmor2003entropy}
\begin{align} \label{burgers_cons_flux}
F_{i+1/2} = \frac{u_i^2 + u_i u_{i+1} + u_{i+1}^2}{6}
\end{align}
which yields a conservative semi-discrete system.
Because of the lack of numerical viscosity,
the semi-discrete solution is not the vanishing-viscosity
solution, but instead develops a dispersive shock.

Energy evolution up to $t=0.2$ is shown in Figures \ref{fig:burgersE}.
As expected, standard RK methods all exhibit significant growth or
dissipation of energy, while RRK methods preserve energy up to roundoff error.
Convergence results at $t=0.03$ (just before shock formation) are shown in
Figure \ref{fig:burgersC}, where we compare the IDT approach with the full relaxation
approach.
All RRK methods achieve the order of accuracy of the corresponding RK method,
whereas for IDT methods the convergence rate is reduced by one, as predicted
by Theorem \ref{orderthm1}.

\begin{figure}
    \centering
    \begin{subfigure}[b]{0.49\textwidth}
        \includegraphics[width=\textwidth]{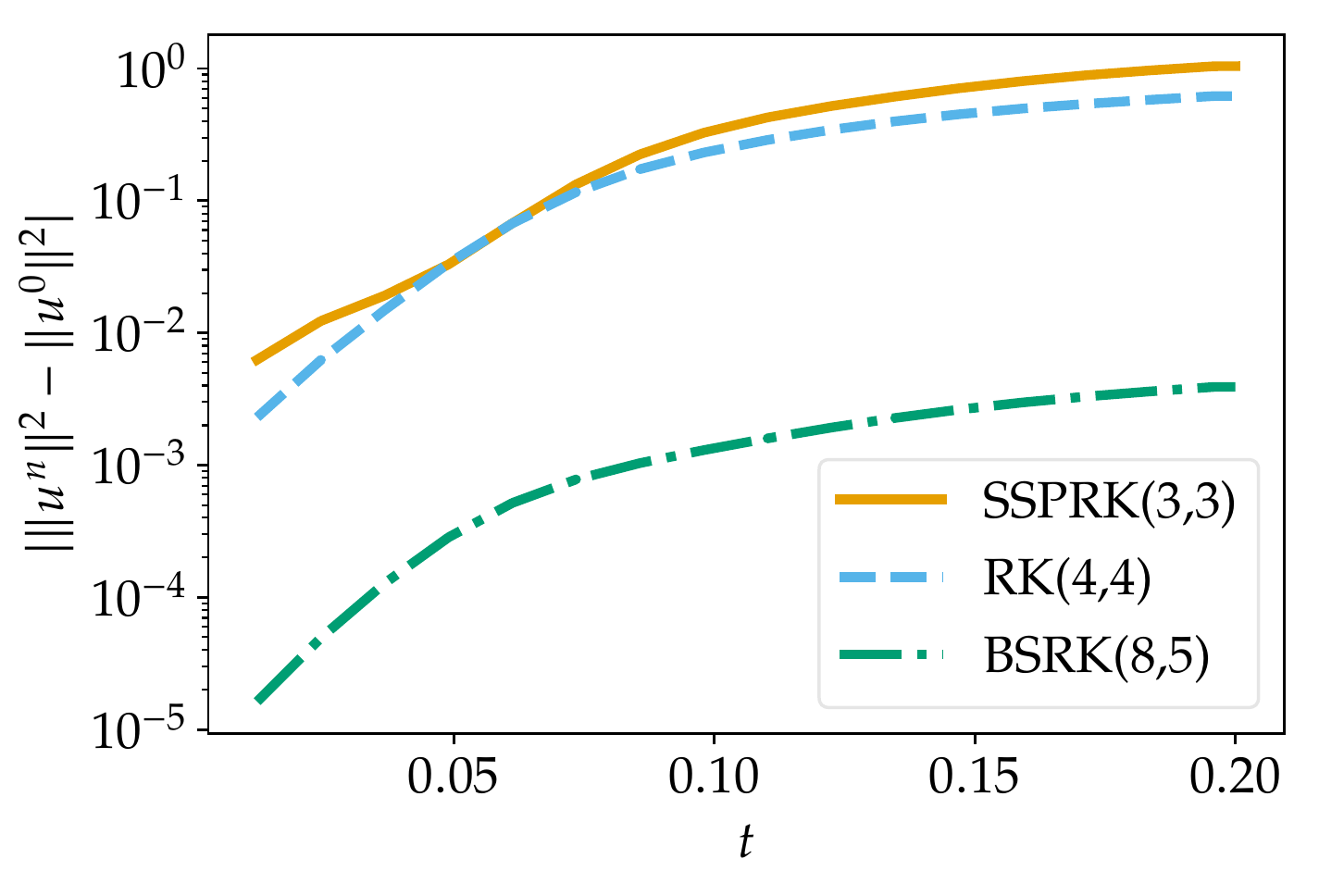}
        \caption{Standard RK methods\label{fig:burgers_energy_RK}}
    \end{subfigure}~
    \begin{subfigure}[b]{0.49\textwidth}
        \includegraphics[width=\textwidth]{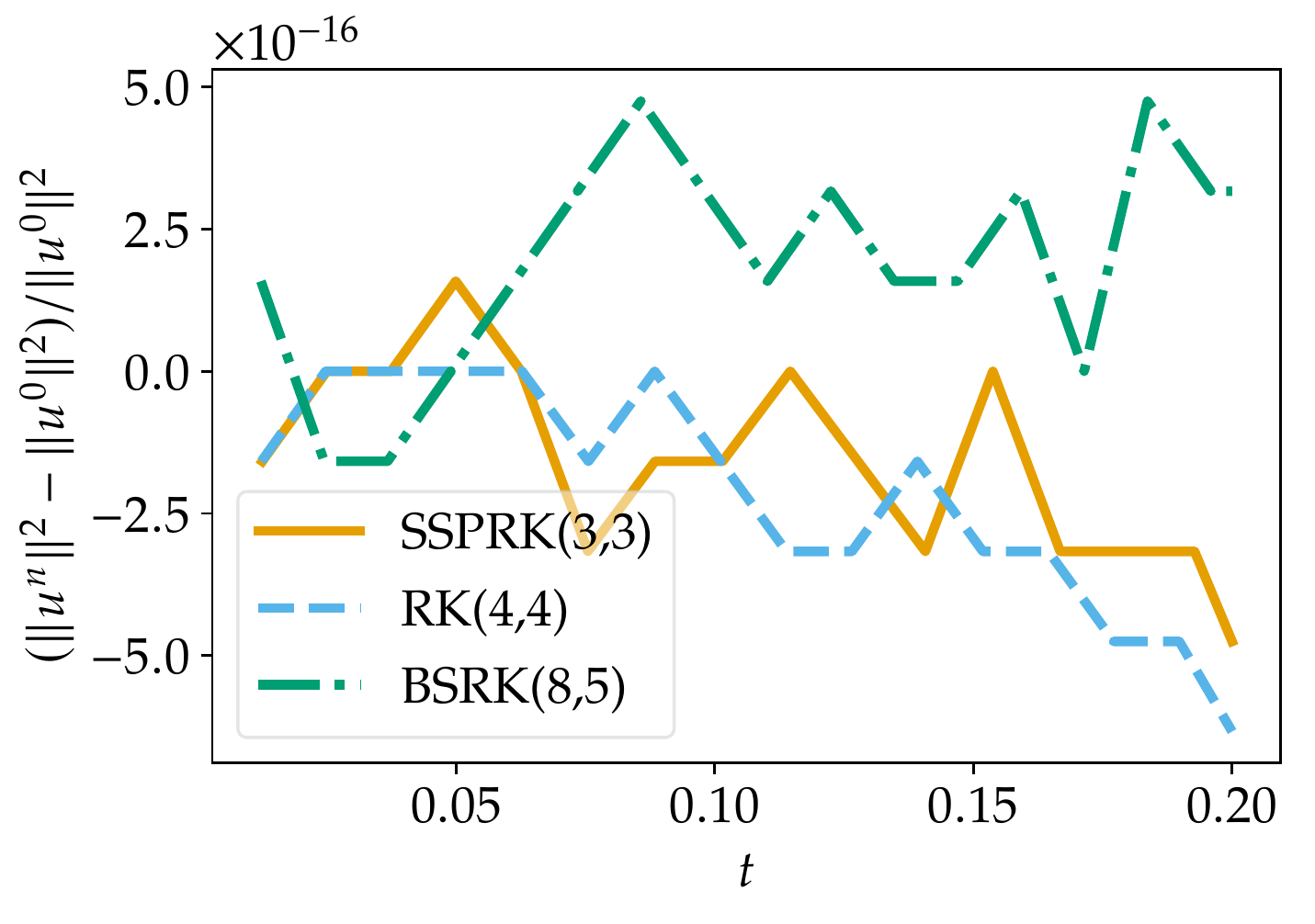}
    \caption{Relaxation RK methods\label{fig:burgers_energy_RRK}}
    \end{subfigure}
    \caption{Evolution of the energy in the numerical solution of Burgers'
            equation \eqref{burgers} with energy-conservative flux \eqref{burgers_cons_flux},
            using $\Dt = 0.3\Delta x$.
            Standard RK methods exhibit energy growth or decay, while relaxation
            methods conserve energy to within roundoff error.  Note that the vertical scale
            of the left figure is logarithmic, while the scale of the right figure
            is linear.
            \label{fig:burgersE}}
\end{figure}

\begin{figure}
    \centering
    \begin{subfigure}[b]{0.49\textwidth}
        \includegraphics[width=\textwidth]{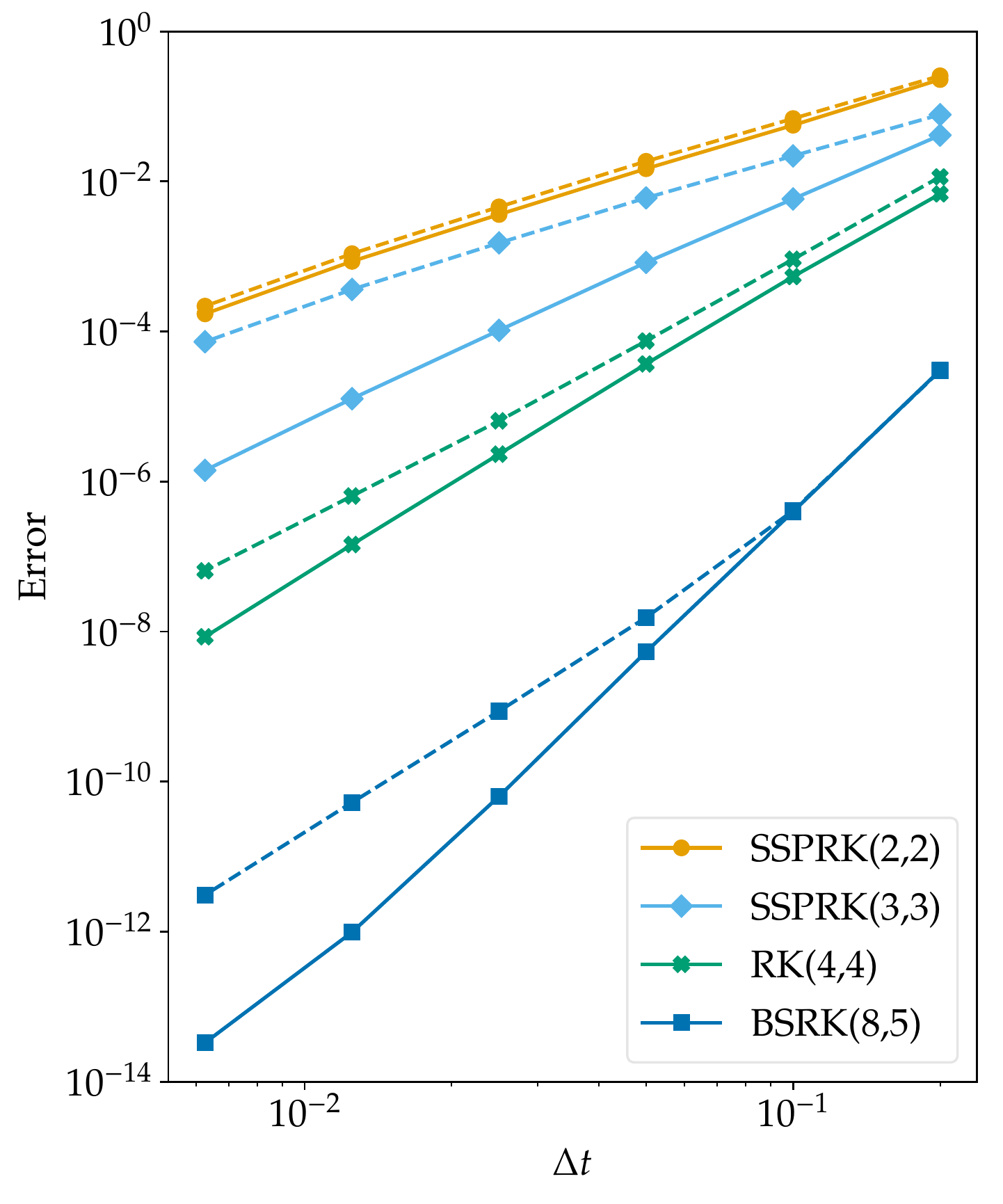}
        \caption{IDT methods\label{fig:burgers_convergence}}
    \end{subfigure}~
    \begin{subfigure}[b]{0.49\textwidth}
        \includegraphics[width=\textwidth]{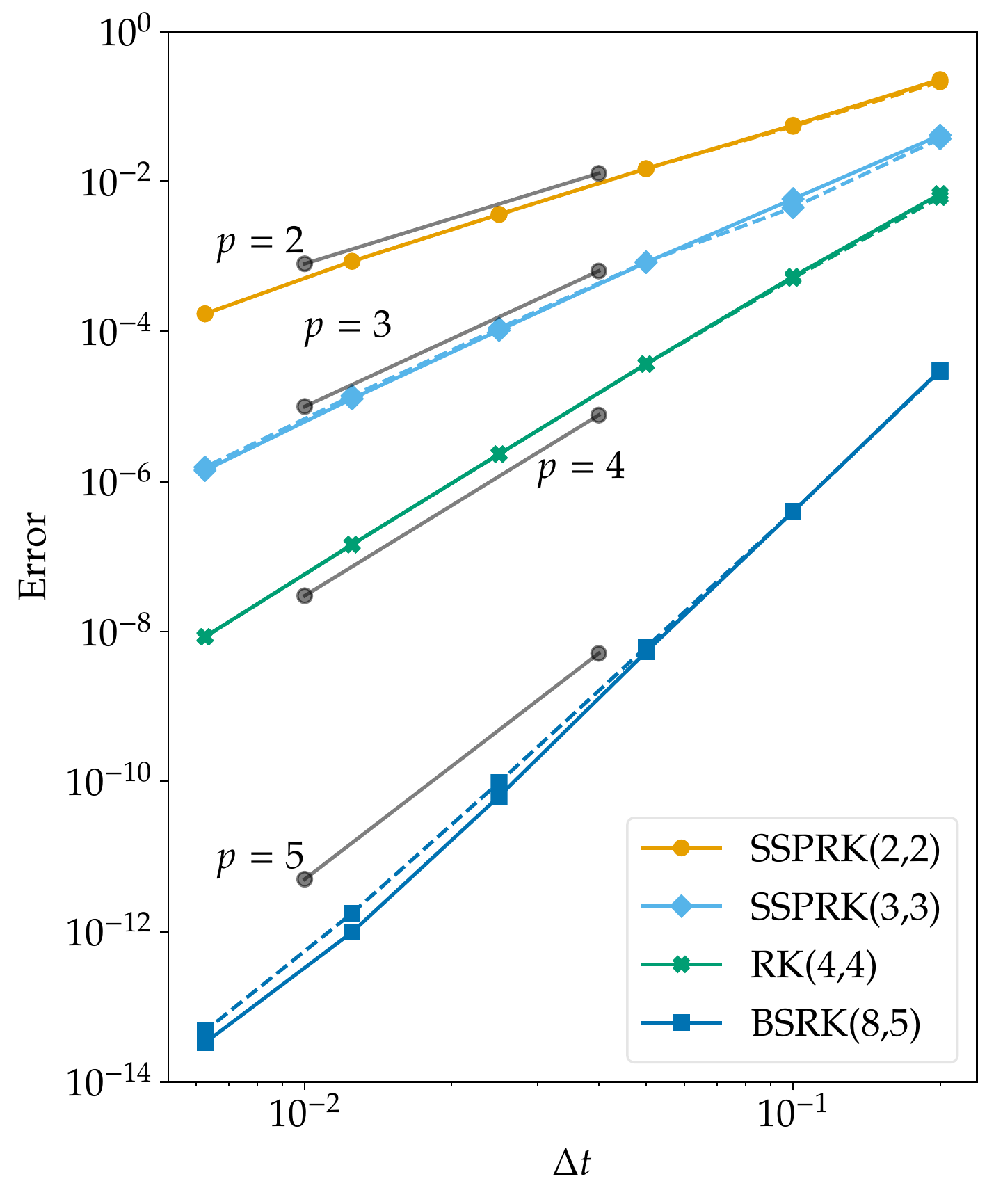}
    \caption{Relaxation methods\label{fig:burgers_convergence_fixedstep}}
    \end{subfigure}
\caption{Convergence for Burgers' equation with energy-conservative flux,
using several standard RK methods (solid lines) and their energy-conserving modifications
(dashed lines).
The value of $\Delta x$ is fixed and the solution is computed at $t=0.03$, just before
the time of shock formation.  Without step size rescaling, the rate of
convergence is reduced by one in some cases as indicated in Theorem \ref{orderthm1}.
With step size rescaling, in most cases the standard and relaxed RK methods give
almost exactly the same accuracy.
\label{fig:burgersC}}
\end{figure}

\subsubsection{Energy-dissipative semi-discretization}
We obtain a dissipative system by adding a centered difference to the flux:
\begin{align} \label{burgers_diss_flux}
F_{i+1/2} = \frac{u_i^2 + u_i u_{i+1} + u_{i+1}^2}{6} - \epsilon(u_{i+1} - u_{i}).
\end{align}
The amount of dissipation
is controlled by $\epsilon >0$.  The scheme is still consistent with \eqref{burgers}
since the amount of dissipation is proportional to $\Delta x$.  We take
$\epsilon = 1/100$.  With this dissipative flux, the solution develops a
viscous shock.

Results are shown in Figures \ref{fig:burgersdissE} and \ref{fig:burgersdissC}.
What is most interesting is that applying the relaxation approach dramatically
improves the numerical approximation of the global dissipation.
Convergence results are similar to those obtained with the conservative flux.

\begin{figure}
    \centering
    \begin{subfigure}[b]{0.49\textwidth}
        \includegraphics[width=\textwidth]{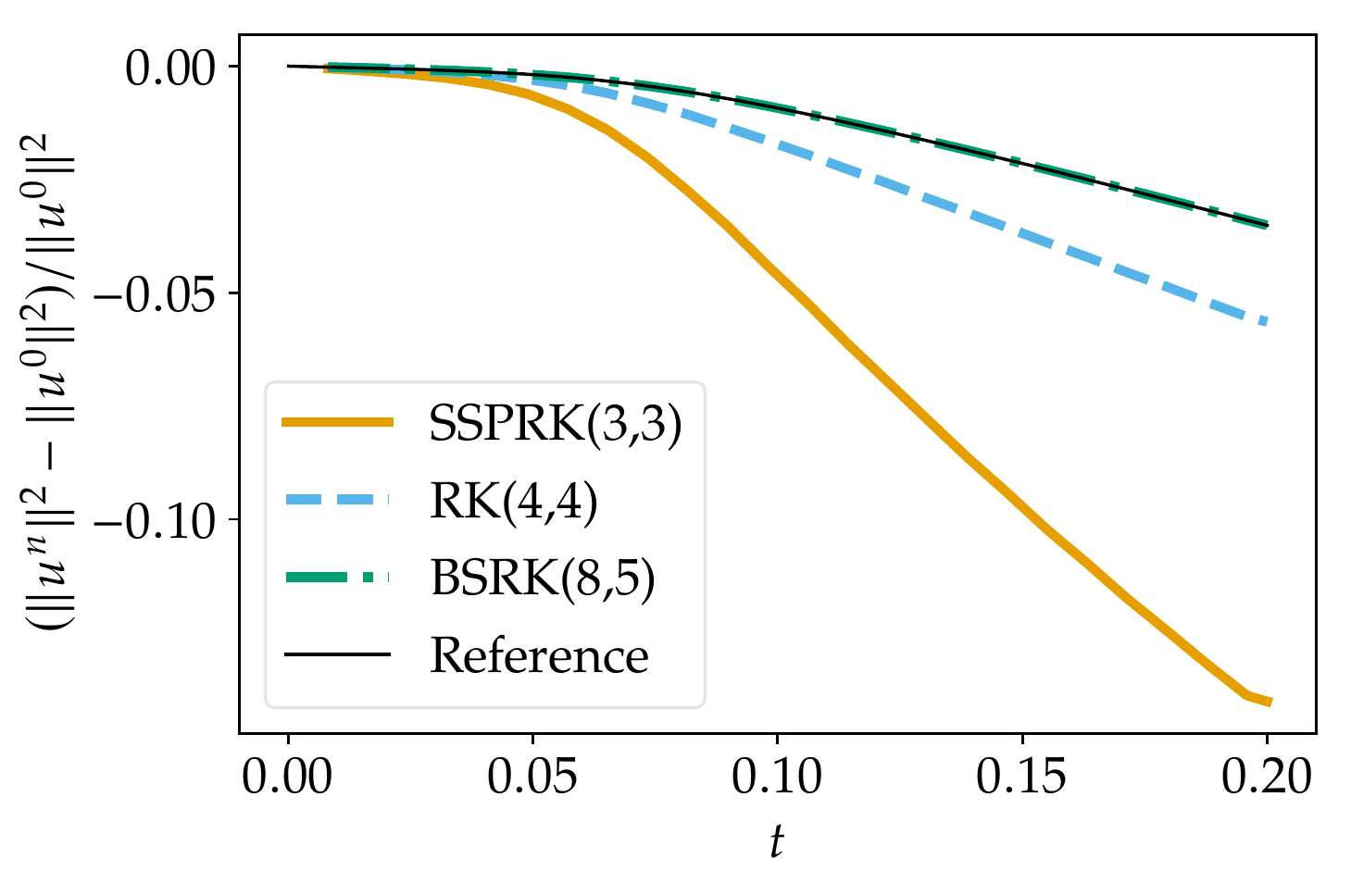}
        \caption{Standard RK methods\label{fig:burgersdiss_energy_RK}}
    \end{subfigure}~
    \begin{subfigure}[b]{0.49\textwidth}
        \includegraphics[width=\textwidth]{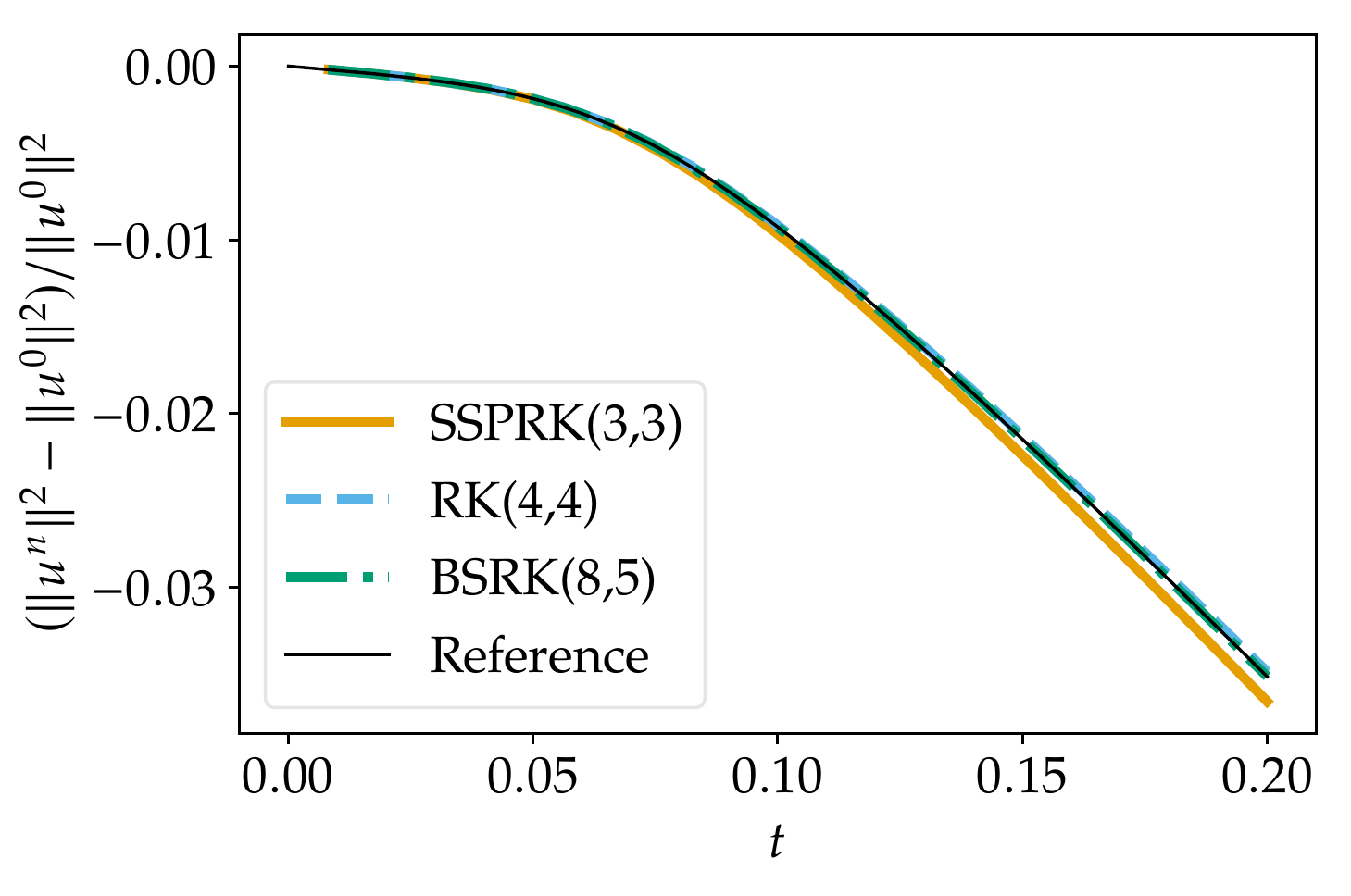}
    \caption{Relaxation RK methods\label{fig:burgersdiss_energy_RRK}}
    \end{subfigure}
    \caption{Evolution of the energy in the numerical solution of Burgers'
            equation \eqref{burgers} with energy-dissipative flux \eqref{burgers_diss_flux},
            with a $\Delta t = 0.2 \Delta x$.
            Standard RK methods (left) exhibit excessive dissipation, due to numerical
            errors.  Relaxation methods (right) approximate the correct energy evolution much more
            accurately.
            \label{fig:burgersdissE}}
\end{figure}

\begin{figure}
    \centering
    \begin{subfigure}[b]{0.49\textwidth}
        \includegraphics[width=\textwidth]{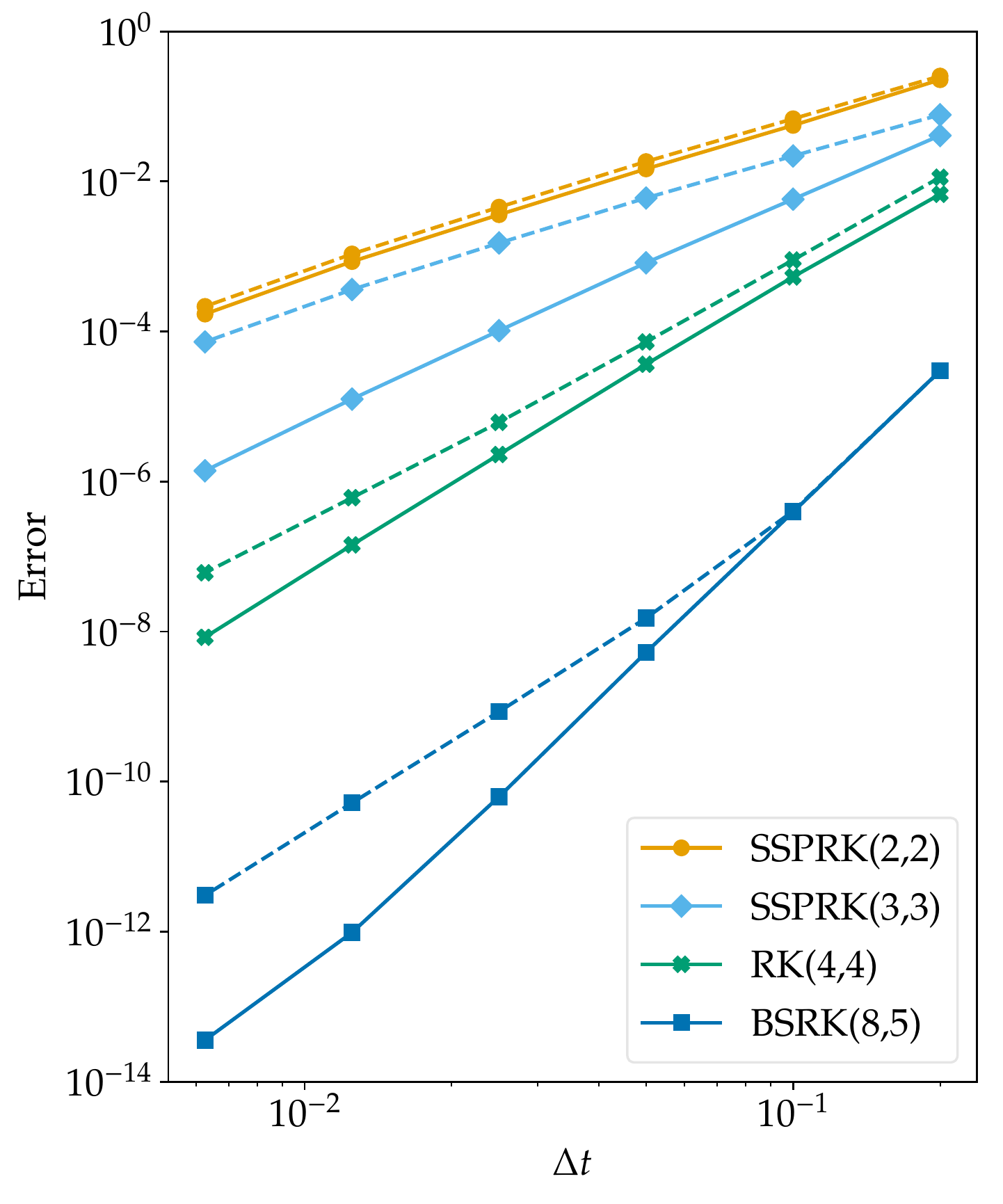}
        \caption{IDT methods\label{fig:burgersdiss_convergence}}
    \end{subfigure}~
    \begin{subfigure}[b]{0.49\textwidth}
        \includegraphics[width=\textwidth]{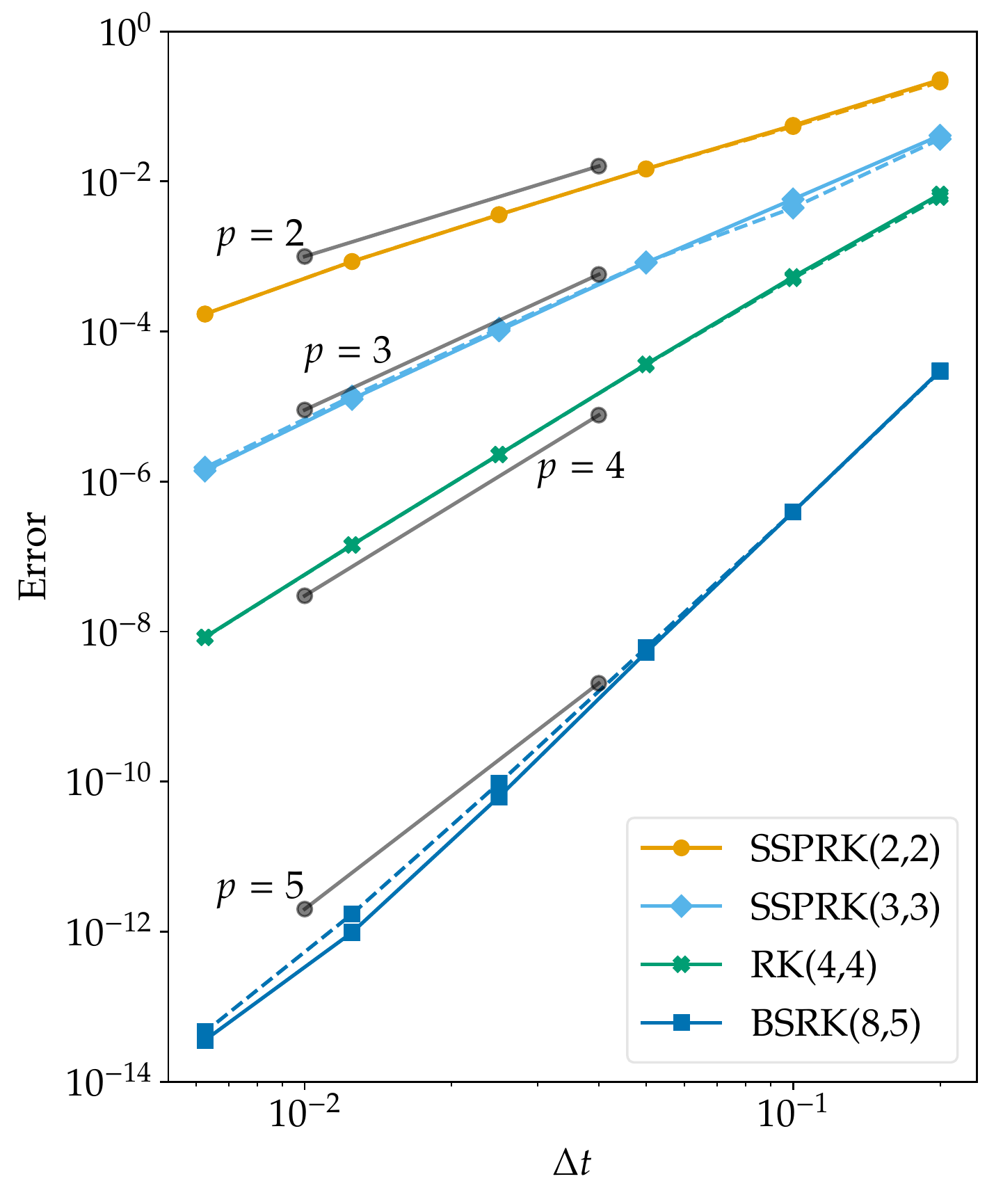}
    \caption{Relaxation methods\label{fig:burgersdiss_convergence_fixedstep}}
    \end{subfigure}
\caption{Convergence for Burgers' equation with energy-dissipative flux,
using several standard RK methods (solid lines) and their energy-conserving modifications
(dashed lines).
The value of $\Delta x$ is fixed.  The IDT methods (i.e., relaxation but without step size rescaling)
give a rate of convergence that is reduced by one in some cases as indicated in Theorem \ref{orderthm1}.
With step size rescaling, in most cases the standard and relaxed RK methods give
similar accuracy.
\label{fig:burgersdissC}}
\end{figure}

%

\section{Conclusions}
The relaxation approach we have proposed seems to be a simple and effective way
to make any Runge--Kutta method preserve conservation or dissipativity with
respect to an inner-product norm.  This can be extended to more general convex
functionals; see \cite{paper2}.  While we have focused here on explicit
methods, the technique applies to implicit methods as well.
Like Runge--Kutta methods, relaxation Runge--Kutta methods automatically
preserve linear first integrals.  If the original method is equipped with an
embedded error estimator or dense output formula, these can be used without
modification (other than the rescaling of the step size when determining
the dense output times).  As we
have shown, many SSP RK methods retain the same SSP coefficient
when used as RRK methods.  The linear stability properties of a method are
slightly modified depending on the choice of $\gamma$, but in general the
allowable step size for an RRK method is essentially the same as that allowed
for the original RK method.

Together these properties make RRK methods an attractive choice for
symmetric hyperbolic systems; in combination with entropy-stable
spatial discretizations they give a fully-discrete, explicit scheme
that is provably entropy stable (for quadratic entropies).

In several experiments we have compared results obtained by viewing
$u_{\gamma}^{n+1}$ as an approximation to $u(t_n+\Dt)$ (the so-called
incremental direction technique, or IDT) versus those
obtained by viewing it as an approximation to $u(t_n +\gamma_n \Dt)$
(the relaxation approach).
The main purpose of this comparison is to illustrate our theoretical
convergence estimates; it is clear that in practice one should always
use the latter interpretation.  Comparison with embedded projection
methods is planned as future work.

For an RK method with $s$ stages, determination of $\gamma_n$ via the formula
\eqref{gammadefreal} requires the evaluation of $s+1$ inner products
(see Remark \ref{ipremark}).
For typical high-order discretizations of nonlinear PDEs this cost is
negligible compared to the $s$ evaluations of $f$ required for each step.
For simpler applications (such as linear wave equations) the cost may be an important factor.
A careful comparison of RRK schemes relative to other time discretizations
for linear wave equations is the subject of ongoing work.

\section{Acknowledgments}
The author is grateful to Hendrik Ranocha for helpful comments on drafts of this
work.  This work was supported by funding from King Abdullah University of Science \& Technology.

\bibliographystyle{plain}
\bibliography{refs}

\end{document}